%% file: main.tex
\documentclass[11pt,reqno, a4paper]{article}

\usepackage[LCY,T1]{fontenc}

\usepackage[russian,main=english]{babel}
\usepackage[utf8]{inputenc}
\usepackage[T1]{fontenc}
\usepackage{siunitx}
\usepackage{pgfplots}
\pgfplotsset{compat=1.18}
\usepackage{algorithm, xpatch}
\usepackage[noend]{algpseudocode}
\usepackage{anyfontsize}
\usepackage{amsmath, amsthm, amssymb}
\usepackage{mathtools}
\usepackage{thmtools}
\usepackage{microtype}
\usepackage{csquotes}
\usepackage{comment}
\usepackage{enumitem}
\usepackage{booktabs}
\usepackage{caption}
\usepackage{xfrac}
\usepackage{scalerel}
\usepackage[a4paper, left=2.5cm, right=2.5cm, top=2.5cm, bottom=3cm, footskip=1.6cm]{geometry}
\numberwithin{equation}{section}
\renewcommand{\theHequation}{\thesection.\arabic{equation}}

\newtheorem{theorem}{Theorem}[section]
\newtheorem{lemma}[theorem]{Lemma}
\newtheorem{corollary}[theorem]{Corollary}
\newtheorem{prop}[theorem]{Proposition}
\newtheorem{remark}[theorem]{Remark}
\newtheorem{definition}[theorem]{Definition}

\usepackage{hyperref}
\usepackage[nameinlink]{cleveref}
\crefname{prop}{proposition}{propositions}
\Crefname{prop}{Proposition}{Propositions}

\newcommand{\N}{\mathbb{N}}
\newcommand{\Z}{\mathbb{Z}}
\newcommand{\R}{\mathbb{R}}
\newcommand{\C}{\mathbb{C}}

\newcommand{\rlasso}{\texttt{rLasso}}
\newcommand{\omp}{\texttt{OMP}}
\newcommand{\cosamp}{\texttt{CoSaMP}}

\DeclareMathOperator{\spann}{span}
\DeclareMathOperator{\supp}{supp}
\DeclareMathOperator*{\argmin}{argmin}
\DeclareMathOperator*{\argmax}{argmax}

\makeatletter
\def\blfootnote{\xdef\@thefnmark{}\@footnotetext}
\xpatchcmd{\algorithmic}{\itemsep\z@}{\itemsep=.4ex}{}{}
\makeatother

\makeatletter
\def\@seccntformat#1{\@ifundefined{#1@cntformat}%
   {\csname the#1\endcsname\quad}  
   {\csname #1@cntformat\endcsname}
}

\let\oldappendix\appendix 
\renewcommand\appendix{%
    \oldappendix
    \newcommand{\section@cntformat}{\appendixname~\thesection\quad}
    \renewcommand{\theHsection}{appendix.\Alph{section}}
    \renewcommand{\theHsubsection}{appendix.\Alph{section}.\arabic{subsection}}
    \renewcommand{\theHsubsubsection}{appendix.\Alph{section}.\arabic{subsection}.\arabic{subsubsection}}
    \renewcommand{\theHequation}{appendix.\Alph{section}.\arabic{equation}}
    \renewcommand{\theHtheorem}{appendix.\Alph{section}.\arabic{theorem}}
    \renewcommand{\theHfigure}{appendix.\Alph{section}.\arabic{figure}}
    \renewcommand{\theHtable}{appendix.\Alph{section}.\arabic{table}}
}
\makeatother

\begin{document}

\title{High-dimensional sparse recovery from function samples --\\
Decoders, guarantees and instance optimality}

\author{Moritz Moeller $\!{}^{a}$,
Sebastian Neumayer $\!{}^{a}$,
Kateryna~Pozharska $\!{}^{a,b}$,\\
Tizian Sommerfeld $\!{}^{a}$,
Tino Ullrich $\!{}^{a,}$\footnote{Corresponding author, Email:
tino.ullrich@math.tu-chemnitz.de}\\\\
${}^{a}\!$ Chemnitz University of Technology, Faculty of Mathematics\\[2mm]
${}^{b}\!$ Institute of Mathematics of NAS of Ukraine}

\date{October 19, 2026}

\maketitle

\begin{abstract} 
We investigate the reconstruction of multivariate functions from samples using sparse recovery techniques.
For Square Root Lasso, Orthogonal Matching Pursuit, and Compressive Sampling Matching Pursuit, we demonstrate both theoretically and empirically that they allow us to recover functions from a small number of random samples.
In contrast to Basis Pursuit Denoising, the deployed decoders only require a search space $V_J$ spanned by dictionary elements indexed by $J$ and a sparsity parameter $n$ to guarantee an $L_2$-approximation error decaying no worse than a best $n$-term approximation error and the truncation error with respect to the search space $V_J$ and the uniform norm.
We show that this happens simultaneously for all admissible functions if the number of samples scales as $n\log^2 n\log |J|$, coming from known bounds for the RIP for matrices built upon bounded orthonormal systems.
As a consequence, we obtain bounds for sampling widths in function classes.
In addition, we establish lower bounds on the required sample complexity, which show that the log-factor in $\vert J \vert$ is indeed necessary to obtain such {\em instance-optimal} error guarantees. 
Finally, we conduct several numerical experiments to show that our theoretical bounds are reasonable and compare the discussed decoders in practice.
\end{abstract} 

\input{sections/intro.tex}

\input{sections/recovery.tex}
\input{sections/numerics.tex}
\input{sections/appendix.tex}

\noindent\textbf{Acknowledgment.} MM and TS are supported by the ESF, being co-financed by the European Union and from tax revenues on the basis of the budget adopted by the Saxon State Parliament. KP would like to acknowledge support by the German Research Foundation (DFG 403/4-1) and the budget program
of the NAS of Ukraine ``Support for the development of priority areas
of research'' (KPKVK 6541230). KP and TU would like to thank Ben Adcock for pointing out reference \cite{AdBrWe22,BDJR21} and bringing square root Lasso as a noise-blind alternative to the authors' attention during a discussion within the session {\em Function Recovery and Discretization Problems} organized by David Krieg and KP at the conference MCQMC24 in Waterloo, Ontario, Canada. 


\fontsize{11}{11}\selectfont
\bibliographystyle{abbrv}

\end{document}

%% file: sections/intro.tex
\section{Introduction} 

Many real-world applications aim to reveal correlations and dependencies in data.
Mathematically, we model the quantity of interest by a function  $f\colon \Omega \to \mathbb{C}$ to which we have access only through point evaluations at sampling points $\mathbf{X}=\{x^1,\ldots,x^m\}\subset\Omega$.
Thus, we have to approximate the unknown function $f$ from its values $f(\mathbf{X})=(f(x^1),\ldots, f(x^m))^{\top}\in\mathbb{C}^m$ at the points $\mathbf{X}$, a process often referred to as decoding.
In practice, these samples $f(\mathbf{X})$ are commonly perturbed by noise.
Let $\mathcal{B} = \{b_j\in C(\Omega)\colon \; j \in I\}$ be a countable $L_2(\mu)$-orthonormal basis of continuous functions. 
We aim to approximate $f$ by (finitely many) atoms $b_j$ indexed by $j \in J \subset I$.
To recover the associated coefficients $\mathbf{c} \in \mathbb{C}^{\lvert J\rvert}$, we may solve the (potentially underdetermined) linear system
\begin{equation}\label{eq:LinearSystem}
    \mathbf{A}\mathbf{c} = \mathbf{y}, \qquad
    \text{where} \quad \mathbf{A} = m^{-\sfrac{1}{2}} \bigl(b_j\bigl( x^k \bigr)\bigr)_{k=1, \ldots, m,\,  j\in J}\,, \quad \mathbf{y} = m^{-\sfrac{1}{2}} \bigl(f\bigl( x^k \bigr)\bigr)_{k=1, \ldots, m}.
\end{equation}
Here, we implicitly assume that the index set $J$ is equipped with an enumeration.
Below, we refer to a (non-linear) solution operator $T_{\mathbf{X}}$ for \eqref{eq:LinearSystem} with $f(\mathbf X) \mapsto \sum_{j\in J} c_j b_j$ as decoder.

The recovery turns out to be particularly easy if $f$ can be well-represented by only a few atoms from the orthonormal basis $\mathcal{B}$.
In this case, we can deploy the non-linear decoders Square Root Lasso (\rlasso) \cite{BeCheWa11}, as well as Orthogonal Matching Pursuit (\omp) \cite{KuRa08} and the related  Compressive Sampling Matching Pursuit (\cosamp) \cite{NeeTr09}.
The greedy decoders \omp \ and \cosamp \ utilize an iterative atom selection process to exploit the intrinsic sparsity of the target function $f$.
Both algorithms require only the search spaces $V_J = \mathrm{span}(b_j\colon \; j \in J)$ and a sparsity parameter $n$ as \enquote{parameters} to produce an output vector with $\mathcal{O}(n)$ nonzero entries.
Following a different strategy, \rlasso \ recovers (approximate) coefficients $\mathbf c$ by solving the convex optimization program
\begin{equation} \label{eq:rLASSOIntro}
    \argmin_{\mathbf{c} \in \mathbb{C}^{\lvert J\rvert}} \bigl\{\lVert\mathbf{c}\rVert_{\ell_1} + \lambda\lVert\mathbf{Ac - \mathbf{y}}\rVert_{\ell_2}\bigr\},
\end{equation}
which is noise-blind: the regularization parameter $\lambda>0$ is a function of $n$, independent of potential noise.

A natural question is whether we can control the approximation error $\lVert f - T_{\mathbf{X}}(f)\rVert_{V}$ with respect to a norm $\lVert\cdot\rVert_V$.
In general, this quantity depends on  the chosen search spaces $V_J$ and the sampling points $\mathbf{X}$.
For the proposed non-linear decoders, the role of $V_J$ turns out to be not as crucial as for (weighted) least-squares decoders \cite{CoMi16,KrUl21,DKU23}, for which one has to know the {\em optimal} subspace $V_J$ for the considered approximation, see also \Cref{rem:lsqr}.
Throughout, we consider independent and identically distributed (iid) points $\mathbf{X}$ sampled from the orthogonality measure $\mu$. 

\paragraph{Contribution}
First, we investigate the conditions, in particular the required sample size $m = \lvert\mathbf{X}\rvert$, for the deployed decoders $T_{\mathbf{X}}$ to achieve an approximation error comparable to the truncation error \smash{$\lVert f-P_Jf\rVert_{L_\infty}$}, where $P_J f$ denotes the $L_2(\mu)$-orthogonal projection $P_J f = \sum_{j\in J} \langle f,b_j\rangle b_j$. 
Our sparse recovery result in \Cref{thm:Error_Recovery_Ops} for \rlasso \ and \omp \ involves the best $n$-term approximation error
\begin{equation}
    \sigma_n(f)_{V}\coloneqq\sigma_n(f;\mathcal{B})_{V} \coloneqq \inf\limits_{\substack{J \subset I\\\lvert J \rvert \leq n}} \,\inf\limits_{\substack{c_j\in\mathbb{C}\\j\in J}}\,\Bigl\lVert f-\sum\limits_{j\in J}c_jb_j\Bigr\rVert_{V}
\end{equation}
with respect to the Wiener space $V = \mathcal{A}_1$. The norm of this space is the sum over absolute values of the Fourier coefficients, defined in \eqref{eq:DefWiener}. This establishes for any $2\leq q\leq \infty$ the universal bound 
\begin{equation}\label{eq:BoundError}
	\lVert f - T_{\mathbf{X}}(f)\rVert_{L_q(\mu)} \leq C_{\mathcal{B}} n^{\sfrac{1}{2}-\sfrac{1}{q}}\bigl(n^{-\sfrac{1}{2}}\sigma_n(f;\mathcal{B})_{\mathcal{A}_1} + \lVert f-P_J f\rVert_{L_\infty}\bigr),
\end{equation}
which holds with high probability for all $f\in \mathcal{A}_1$ provided that the number of samples $m$ scales as 
\begin{equation}\label{eq:sampl}
    m\gtrsim n\log^2 n\log\lvert J\rvert.
\end{equation}
The sample complexity bound \eqref{eq:sampl} relies on recent improved bounds for RIP matrices based on bounded orthonormal systems, see \cite{HaRe17}, \cite{AdBrWe22}.
Note that the dependence of the bound \eqref{eq:BoundError} on the best $n$-term approximation error of $f$ relates to \emph{instance optimality}, a term introduced by Cohen, Dahmen and DeVore in \cite{CDD09}.
We extend \Cref{thm:Error_Recovery_Ops} to \cosamp \ in \Cref{thm:Error_Recovery_Cosamp}, and moreover we discuss the case of unbounded systems in \Cref{Legendre}. 
Building upon techniques from combinatorics \cite{FPRU10, FoRa13}, we also establish a lower bound on the number of required sampling points for \emph{any} decoder, see \Cref{instance_B1B1}:
If \eqref{eq:BoundError} is to hold for all $f \in \mathcal A_1$, then we must have
\begin{equation}
    m \gtrsim n\log (\lvert J\rvert / 4n ).
\end{equation}
Consequently, the $\log\lvert J \rvert$-factor in \eqref{eq:sampl} cannot be avoided.
Finally, we present variants of our results for trigonometric approximation in \Cref{sec:Fourier}, highlighting their implications for sampling complexity in function spaces.

To illustrate our theoretical observations, we conduct large-scale recovery experiments, for which we focus on the widely used Fourier and Chebyshev systems as $\mathcal{B}$.
While Legendre polynomials are theoretically admissible by the results of \Cref{Legendre}, their efficient numerical implementation turned out infeasible.
To control the size of the search space $V_J$, we rely on hyperbolic crosses \cite{DTU18}.
For \omp, we solve the least-squares problem via an iterative update of the Cholesky decomposition \cite{StuChr2012}, while we use the iterative least-squares solver LSQR for \cosamp.
To solve the \rlasso \ optimization program \eqref{eq:rLASSOIntro}, we employ a restarted primal-dual algorithm based on \cite{AdcCol2025}.
Although the theoretical approximation bounds coincide, our numerical experiments show that \rlasso \ and \cosamp \ are preferable when the number of coefficients required for a given approximation accuracy is large.
This is unsurprising, since each \omp \ update adds only a single support point at the cost of solving a linear system.
On the other hand, our experiments show that both \omp \ and \cosamp \ are superior to \rlasso \ for highly sparse solutions, i.e., if we have to perform only few iterations.
Another interesting observation is that we frequently outperform the theoretically guaranteed asymptotic rates, indicating favorable pre-asymptotic behavior:
the decoders exhibit an error decay rate consistent with the numerically computed best 
$n$-term approximation, thereby aligning with the theoretical bound in \eqref{eq:BoundError}.
This represents a substantial advancement over previous results on least-squares recovery reported in \cite{BarLutNag2023} and \cite{KamUllVol2021}.

\paragraph{Related work and broader impact.}
The analysis of \omp \ \cite{DaTe24,KuRa08,Zha11} requires the restricted isometry property (RIP) of the matrix $\mathbf A$ at the sparsity level $n$, see \cite[Chapt.\ 6]{FoRa13} and \cite{Zha11}.
For the reader's convenience, we collect all required results in a self-contained form in \hyperref[app:RIPNSP]{Appendix \ref*{app:RIPNSP}}.
The theoretical properties of \rlasso \ have likewise received considerable attention \cite{ABB19,AdBrWe22,BeCheWa11,PeJu22}.
Note that the recent work \cite{HerNeu2025} replaces the sparsity-promoting regularization \eqref{eq:rLASSOIntro} by sparse optimization performed directly over the orthonormal basis $\mathcal B$, thereby avoiding the need to define $V_J$.

A complementary line of research quantifies the best achievable worst-case recovery error over an entire function space $\mathbf{F}\subset V$ (instead of the single sample in \eqref{eq:BoundError}) by the (non-linear) sampling widths
\begin{equation}\label{def:nonlin_sampling_numbers}
    \varrho_m (\mathbf{F})_V \coloneqq
    \inf_{\mathbf{X}= \{x^1, \dots, x^{m}\}}\,
    \inf_{T\colon \mathbb{C}^m \to V}\,
    \sup\limits_{\smash{\lVert f \rVert_{\mathbf{F}} \leq 1}}
    \lVert f - T(f(\mathbf{X})) \rVert_{V},
\end{equation}
and, when $T$ is required to be linear, by the linear sampling widths $\varrho^{\rm lin}_m (\mathbf{F})_V$.
Our analysis feeds directly into this framework:
since the sampling points
$\mathbf{X}$ in \eqref{eq:BoundError} are admissible for all $f$ simultaneously, we obtain a relation between $\varrho_m$ and the best $n$-term width
$\sigma_n(\mathbf{F};\mathcal{B})_{\mathcal{A}_1} \coloneqq \sup_{\Vert f \Vert_{\mathbf{F}} \leq 1}\sigma_n(f;\mathcal{B})_{\mathcal{A}_1}$, see \Cref{rem:FuncClass}.
Combined with known estimates for best $n$-term widths \cite{DTU18, MPU25_1},
this implies that the asymptotic decay of \smash{$\varrho_m (\mathbf{F})_{L_2}$} is strictly faster than that of \smash{$\varrho^{\rm lin}_m (\mathbf{F})_{L_2}$} for a range of spaces
$\mathbf{F}\subset L_2(\mu)$, among them the weighted Wiener spaces
$\mathcal{A}^r_\theta(\mathcal{B}) \subset \mathcal{A}_1(\mathcal{B})$ with
$\theta\leq 1$ and $r>0$ \cite{JUV23, MSU25, NNS22} and certain Sobolev--Besov
spaces of mixed smoothness \cite{MSU25}.
Thus, in these settings, the studied non-linear decoders provably surpass every linear decoder.

\paragraph{Outline.} 
\Cref{sec:RecoveryMethods} defines \rlasso \ and \omp \ for our setting. The methods are analyzed, which leads  to the error bounds in \Cref{thm:Error_Recovery_Ops} for the case of bounded orthonormal systems.
In  \Cref{thm:Recovery_General_Bases}, certain unbounded orthonormal systems are discussed, for which we present the Legendre polynomials as a concrete example. 
\Cref{sec:InstanceOptimality} is concerned with providing generalized versions of known compressed sensing results before using them to show lower bounds on the minimal number of samples required for function recovery. 
In \Cref{sec:Fourier}, we look at the special case of the Fourier system. \Cref{sec:Numerics} is dedicated to numerical experiments, their implementation and execution and  their evaluation.

\paragraph{Notation.} 
We denote by $\log(a)$ the natural logarithm of $a>0$ and define $[N]\coloneqq\{1,\ldots,N\}$ for $N \in \N$.
Throughout, vectors and matrices are typeset in boldface and equipped with the $\ell_p$-norms $\lVert\mathbf{v}\rVert_{\ell_p}\coloneqq (\sum_{k=1}^{N}\lvert v_k\rvert^p)^{\sfrac{1}{p}}$, $p\in[1,\infty)$.
Further, we define the \enquote{$\ell_0$-norm} as the number of non-zero entries of $\mathbf{v}$, namely $\lVert\mathbf{v}\rVert_{\ell_0}\coloneqq\lvert\{k\in[N]\colon \; v_k\neq 0\}\rvert$.
Given $\mathbf{v}\in\mathbb{C}^N$ and an index set $S\subset[N]$, we define $\mathbf{v}_S\coloneqq[v_k\cdot\mathbf{1}_S(k)]_{k=1}^N$, i.e.\ we set all entries with indices outside of $S$ to zero.
For $\mathbf{v}\in\mathbb{C}^N$ and $p\in[1,\infty]$, we define the best $n$-term approximation error with respect to $\lVert\cdot\rVert_{\ell_p}$ by
\begin{equation}
\sigma_n(\mathbf{v})_{\ell_p}\coloneqq \inf_{\lVert\mathbf{z}\rVert_{\ell_0}\leq n}\lVert\mathbf{v}-\mathbf{z}\rVert_{\ell_p} = \biggl(\sum\limits_{k=n+1}^N \lvert v_k^{\ast}\rvert^p\biggr)^{\sfrac{1}{p}},
\end{equation}
where $\mathbf{v}^{\ast}=(v_k^{\ast})_{k\in[N]}$ denotes the non-increasing rearrangement of $\mathbf{v}$ with respect to absolute values.
This concept can be extended to infinite sequences in $\ell_p$.

For a topological space $\Omega$ equipped with some \emph{probability} measure $\mu$, let \smash{$L_p(\Omega, \mu)$} denote the Lebesgue spaces of $p$-integrable complex-valued functions with norm \smash{$\lVert f\rVert_{L_p}\coloneqq(\int_\Omega\lvert f\rvert^p\mathrm{d}\mu)^{\sfrac{1}{p}}$} for $p\in[1,\infty)$, and of essentially bounded functions with norm $\lVert f\rVert_{L_\infty}\coloneqq\operatorname{ess\, sup}_{x\in\Omega}\lvert f(x)\rvert$.
If $\mathcal{B}$ is an orthonormal basis of $L_2(\Omega,\mu)$, then there exist unique coefficients $[f] = (\langle f, b_j \rangle)_{j\in I}$ with $f=\sum_{j\in I}[f]_jb_j$, and we have $\sigma_n(f;\mathcal{B})_{\mathcal{A}_\theta}= \sigma_n([f])_{\ell_\theta}$ for any $f\in L_1(\Omega,\mu)$. 
The space $C(\Omega)$ denotes the continuous complex-valued functions on $\Omega$.

For two sequences of non-negative real numbers $(a_n)_n$ and $(b_n)_n$ we write 
$a_n \lesssim b_n$ if there exists a constant $c>0$ such that 
$a_n \leq cb_n$ for all $n\in \mathbb{N}$.
Equivalently, we can use the Landau notation $a_n = \mathcal{O}(b_n)$.
We define $\gtrsim$ analogously, and write $a_n \asymp b_n$ if both hold simultaneously.
Finally, we require the notation $a_n = o(b_n)$, which means that $b_n>0$ for almost all $n\in \mathbb{N}$ and $\lim_{n\to \infty} a_n/b_n = 0$.

%% file: sections/recovery.tex
\section{Recovery guarantees via \rlasso \ and \omp}\label{sec:RecoveryMethods}

We consider the two non-linear decoders \emph{Square Root Lasso} (\rlasso) and \emph{Orthogonal Matching Pursuit} (\omp).
For the function recovery based on an orthonormal basis $\mathcal{B}$, we fix a finite index set $J\subset I$, and define the matrices
\begin{gather}\label{eq:matrix}
	\mathbf{A} \coloneqq m^{-\sfrac{1}{2}}\left(b_j(x^l)\right)_{1 \leq l \leq m, j \in J} \in \mathbb{C}^{m\times \lvert J\rvert}, \quad \mathbf{D} \coloneqq \operatorname{diag}\bigl(\lVert \mathbf{A}_j\rVert_{\ell_2}^{-1}\bigr)_{j\in J} \quad \text{and} \quad \mathbf{B} \coloneqq \mathbf{A}\mathbf{D}
\end{gather}
and the vector $\mathbf{y} \coloneqq m^{-\sfrac{1}{2}} f(\mathbf{X}) \in \mathbb{C}^m$. 
Here, $\mathbf{A}_j$ denotes the $j$-th column of $\mathbf{A}$, which makes $\mathbf{B}$ be the matrix $\mathbf{A}$ with $\ell_2$-normalized columns. 
First, we specify \rlasso.

\begin{definition}[\rlasso]\label{def:square_root_lasso}
For $\lambda>0$, we define $R_{\mathbf{X}}^\lambda\colon C(\Omega) \to V_J$ by
\begin{gather}\label{eq:Lasso_Decoder}
	R_{\mathbf{X}}^\lambda(f) = \sum_{j \in J} r_j(\mathbf{y}) b_j,
\end{gather}
where $\mathbf{r}(\mathbf{y}) = [R_{\mathbf{X}}^\lambda(f)] \in \mathbb{C}^{\lvert J\rvert}$ is a (fixed) solution of the \rlasso \ minimization problem
\begin{gather}\label{eq:MinimizationProblem}
	\min_{\mathbf{z} \in \mathbb{C}^{\lvert J\rvert}} \lVert\mathbf{z}\rVert_{\ell_1} + \lambda \lVert\mathbf{A} \mathbf{z} - \mathbf{y}\rVert_{\ell_2}
\end{gather}
with the matrix $\mathbf{A}$ from \eqref{eq:matrix}, where we implicitly assume that there is a deterministic selection rule in the case of multiple solutions.
\end{definition}

\begin{definition}[\omp]\label{def:OMP} 	
For $n\in \mathbb{N}$, we define $P_{\mathbf{X}}^n\colon C(\Omega) \to V_J$ by
\begin{gather}\label{eq:OMP_Decoder}
    P_{\mathbf{X}}^{n}(f) = \sum_{j \in J} p_j^n(\mathbf{y}) b_j,
\end{gather}
where $\mathbf{p}^n(\mathbf{y}) =[P_{\mathbf{X}}^n(f)]\in\mathbb{C}^{\lvert J\rvert}$ is recursively defined through
\begin{align}
    S^{k+1}(\mathbf{y}) & = S^k \cup \operatorname{argmax} \bigl\{\lvert(\mathbf{B}^\ast (\mathbf{y} - \mathbf{B}\mathbf{p}^k(\mathbf{y})))_j\rvert\colon \;  1 \leq j \leq \lvert J\rvert\bigr\},\label{eq:OMP_iteration1}\\
    \mathbf{p}^{k+1}(\mathbf{y}) & = \operatorname{argmin}\big\{\lVert \mathbf{y} - \mathbf{B}\mathbf{p}\rVert_{\ell_2}\colon \; \supp(\mathbf{p}) \subset S^{k+1}(\mathbf{y})\big\}\label{eq:OMP_iteration2}
\end{align}
with $S^0(\mathbf{y}) \coloneqq \emptyset$ and $\mathbf{p}^0(\mathbf{y}) \coloneqq \mathbf{0}$.
If the $\argmax$ is not a singleton, we select the smallest index $j$.
\end{definition}

\subsection{Analysis of \rlasso \ and \omp}

To derive recovery guarantees for \rlasso \ and \omp, we require two robust recovery results from the literature, which involve the robust null space property and the restricted isometry property of a matrix as defined below.

\begin{definition}[NSP and RIP]
Given a matrix $\mathbf{A}\in\mathbb{C}^{m\times N}$ with $m<N$ and $n\in\N$, we say
\begin{itemize}[itemsep=0pt, topsep=2pt, leftmargin=*]
    \item that $\mathbf{A}$ satisfies the $\ell_q$-robust null space property (NSP) of order $n<N$ for a norm $\lVert\cdot\rVert$ on $\mathbb{C}^m$ if there exist $\varrho\in(0,1)$ and $\tau>0$ such that
    \begin{equation}\label{eq:NSP}
        \lVert\mathbf{v}_S\rVert_{\ell_q} \leq \varrho n^{\sfrac{1}{q}-1}\lVert\mathbf{v}_{S^\mathrm{C}}\rVert_{\ell_1} + \tau \lVert{\mathbf{A}}\mathbf{v}\rVert
    \end{equation}
    for all $\mathbf{v}\in\mathbb{C}^N$ and all $S\subset[N]$ with $\lvert S\rvert\leq n$.
    In this case, we write $\mathbf{A}\in\operatorname{NSP}(n,q,\lVert\cdot\rVert,\varrho,\tau)$.
    \item that $\mathbf{A}$ satisfies the restricted isometry property (RIP) of order $n$ with RIP constant $\delta_n = \delta_n(\mathbf{A})$ if we have
    \begin{equation}\label{eq:RIP}
        (1-\delta_{n})\lVert\mathbf{v}\rVert_{\ell_2}^2 \leq \lVert\mathbf{A}\mathbf{v}\rVert_{\ell_2}^2 \leq (1+\delta_{n})\lVert\mathbf{v}\rVert_{\ell_2}^2
    \end{equation}
    for all $\mathbf{v} \in \mathbb{C}^N$ with $\lVert\mathbf{v}\rVert_{\ell_0} \leq n$.
\end{itemize}
\end{definition}

\begin{remark}\label{rem:RIPadjust}
    The definition of RIP implies that, if a matrix $\mathbf{A}$ satisfies RIP of order $n$ with constant $\delta_n(\mathbf{A})$, then we know that \smash{$1-\delta_n(\mathbf{A}) \leq \lVert \mathbf{A}_j\rVert_{\ell_2}^2 \leq 1+\delta_n(\mathbf{A})$} for all $j\in[N]$. 
    To see this, apply the definition to the standard basis vectors of \smash{$\mathbb{C}^N$}. 
    We now define the matrix \smash{$\mathbf{D} = \operatorname{diag}(\lVert\mathbf{A}_j\rVert_{\ell_2}^{-1})_{j\in[N]}$} and consider $\mathbf{B} = \mathbf{A}\mathbf{D}$, which is $\mathbf{A}$ with $\ell_2$-normalized columns. 
    For an $s$-sparse vector $\mathbf{v}\in\mathbb{C}^N$ we can thus estimate
    \begin{align}
    \begin{split}
        \lVert\mathbf{B}\mathbf{v}\rVert_{\ell_2}^2 &= \lVert \mathbf{A}\mathbf{D}\mathbf{v}\rVert_{\ell_2}^2 \leq (1 + \delta_n(\mathbf{A})) \lVert \mathbf{D}\mathbf{v}\rVert_{\ell_2}^2 \leq (1+\delta_n(\mathbf{A})) \lambda_{\mathrm{max}}^2(\mathbf{D}) \lVert \mathbf{v} \rVert_{\ell_2}^2\leq \frac{1+\delta_n(\mathbf{A})}{1-\delta_n(\mathbf{A})} \lVert\mathbf{v}\rVert_{\ell_2}^2,\\[2mm]
        \lVert\mathbf{B}\mathbf{v}\rVert_{\ell_2}^2 &= \lVert \mathbf{A}\mathbf{D}\mathbf{v}\rVert_{\ell_2}^2 \geq (1 - \delta_n(\mathbf{A})) \lVert \mathbf{D}\mathbf{v}\rVert_{\ell_2}^2 \geq (1-\delta_n(\mathbf{A})) \lambda_{\mathrm{min}}^2(\mathbf{D}) \lVert \mathbf{v} \rVert_{\ell_2}^2 \geq \frac{1-\delta_n(\mathbf{A})}{1+\delta_n(\mathbf{A})} \lVert\mathbf{v}\rVert_{\ell_2}^2.
    \end{split}
    \end{align}
    Elementary calculations now show that the RIP constant of $\mathbf{B}$ is $\delta_{n}(\mathbf{B}) \leq 2\delta_n(\mathbf{A}) (1-\delta_n(\mathbf{A}))^{-1}$. 
    This idea can be found in \cite[Proof of Theorem 3.7]{BDJR21} and is the basis of their estimates regarding \omp.
\end{remark}

The following relation between RIP and NSP can be found in the literature \cite[Thm.\ 6.26]{AdBrWe22}.
We state it in a different form that is better suited for our theory.
A short proof is given in \hyperref[app:RIPNSP]{Appendix \ref*{app:RIPNSP}} for convenience. 

\begin{theorem}[RIP implies $\ell_2$-robust NSP]\label{thm:RIP_implies_NSP}
If $\mathbf{A} \in \C^{m\times N}$ satisfies the RIP of order $2n$ with $\delta_{2n}<\sfrac{1}{3}$, see \eqref{eq:RIP}, then it satisfies the $\ell_2$-robust NSP of order $n$, i.e.\ $\mathbf{A}\in\operatorname{NSP}(n,2,\lVert\cdot\rVert_{\ell_2},\varrho,\tau)$, where the constants $\varrho \in (0,1)$ and $\tau>0$ depend only on $\delta_{2n}$.
\end{theorem}

Now, we specify the robust recovery guarantees for \rlasso \ from \cite[Thm.\ 3.1]{PeJu22} for our setting.
For convenience, the original result is included in \hyperref[app:RIPNSP]{Appendix \ref*{app:RIPNSP}} as \Cref{thm:LassoNSP}. Similar recovery guarantees also hold for \omp, stated after the \rlasso \ version, and can be obtained directly from the literature without significant modifications.

\begin{prop}\label{prop:rLasso_for_NSP} Let $\mathbf{A} \in \mathbb{C}^{m\times N}$ satisfy $\mathbf{A}\in\operatorname{NSP}(n,2,\lVert\cdot\rVert_{\ell_2},\varrho,\tau)$.
Then for all $\mathbf{y} \in \mathbb{C}^m$ and $\mathbf{v} \in \mathbb{C}^{N}$, any solution $\mathbf{r}(\mathbf{y}) \in \mathbb{C}^{N}$ of the \rlasso \ minimization problem \eqref{eq:MinimizationProblem} with $\lambda = 3\tau n^{\sfrac{1}{2}}$ satisfies
\begin{equation}\label{eq:rLasso_NSP_lq}
	\lVert\mathbf{v}- \mathbf{r}(\mathbf{y}) \rVert_{\ell_q} \leq \beta\bigl(n^{\sfrac{1}{q}-1}\sigma_n(\mathbf{v})_{\ell_1} + n^{\sfrac{1}{q}-\sfrac{1}{2}}  \lVert\mathbf{A}\mathbf{v}-\mathbf{y}\rVert_{\ell_2}\bigr)
\end{equation}
for all $q\in[1,2]$ and a constant $\beta>0$ that depends only on $\varrho$ and $\tau$.
\end{prop}

\begin{proof}
    For $q\in\{1,2\}$, equation \eqref{eq:rLasso_NSP_lq} is a direct consequence of \Cref{thm:LassoNSP} if we set $\lambda = 3\tau n^{\sfrac{1}{2}}$ satisfying \eqref{eq:lambdaLowerBound} independent of $\varrho$ and insert it into \eqref{eq:LassoNSP}.
    For $q\in(1,2)$, we use interpolation.
\end{proof}

\begin{prop}[see {\cite[Thm.\ 6.25]{FoRa13}}]\label{prop:OMP_for_RIP}
    Let $\mathbf{A}\in\mathbb{C}^{m\times \lvert J\rvert}$ have the RIP with constant $\delta_{26n}<\sfrac{1}{13}$.
    Then there exists a constant $\beta>0$ depending only on $\delta_{26n}$ such that for any $\mathbf{y}\in\mathbb{C}^m$ the vector $\mathbf{p}^{24n}(\mathbf{y})$ defined through the \omp \ iteration \eqref{eq:OMP_iteration1}, \eqref{eq:OMP_iteration2} satisfies
    \begin{equation}
        \lVert\mathbf{v}-\mathbf{p}^{24n}(\mathbf{y})\rVert_{\ell_q}\leq \beta\bigl(n^{\sfrac{1}{q}-1}\sigma_n(\mathbf{v})_{\ell_1}+n^{\sfrac{1}{q}-\sfrac{1}{2}}\lVert\mathbf{A}\mathbf{v}-\mathbf{y} \rVert_{\ell_2}\bigr)
    \end{equation}
    for all $\mathbf{v}\in\mathbb{C}^{\lvert J\rvert}$ and $q\in[1,2]$.
\end{prop}

Note that the assumption that the matrix $\mathbf{A}$ must have $\ell_2$-normalized columns is missing in \cite[Thm.\ 6.25]{FoRa13}.
However, the proof of this theorem in the book requires \cite[Lemma\ 3.3]{FoRa13}, which needs this assumption, as stated later in the errata to this book, see \cite{ErrataFouRau13}. The first part of the statement \cite[Thm.\ 6.25]{FoRa13} was later given in \cite[Theorem 3.7]{BDJR21} without the assumption that $\mathbf{A}$ has $\ell_2$-normalized columns. The second part, the one which is important for our theory, is stated and proved here. The only difference to the original statement is that we require $\delta_{26n}(\mathbf{A}) < \sfrac{1}{13}$ instead of $\delta_{26n}(\mathbf{A}) < \sfrac{1}{6}$, since \omp \ requires that $\delta_{26n}(\mathbf{B}) < \sfrac{1}{6}$, see \Cref{rem:RIPadjust}.

\begin{proof}
    For this proof, we abbreviate $\mathbf{p} = \mathbf{p}^{24n}(\mathbf{y})$ and define $\mathbf{v}_{2n}\in\mathbb{C}^N$ to be the vector with the $2n$ largest entries of $\mathbf{v}$ in absolute value (it does not matter which ones if they are not unique), and all other entries set to zero. 
    Moreover, let $\overline{\mathbf{v}}_{2n} = \mathbf{v} - \mathbf{v}_{2n}$ and $\mathbf{e}$ be the error term defined by $\mathbf{y} = \mathbf{A}\mathbf{v} + \mathbf{e}$. 
    At this point, we use \cite[Theorem 3.7]{BDJR21}, which says
    \begin{equation}\label{eq:UpperBoundOMPVector}
        \lVert \mathbf{A}\mathbf{p} - \mathbf{y}\rVert_{\ell_2} \leq 49 \inf\limits_{\substack{\mathbf{z}\in\mathbb{C}^N\\\lvert\supp(\mathbf{z})\rvert\leq 2n}} \lVert \mathbf{A}\mathbf{z}-\mathbf{y}\rVert_{\ell_2},
    \end{equation}
    and thus enables us to say $\lVert \mathbf{A}\mathbf{p} - \mathbf{y}\rVert_{\ell_2} \leq 49 \lVert \mathbf{A}\mathbf{v}_{2n}-\mathbf{y}\rVert_{\ell_2} = 49\lVert \mathbf{A}\overline{\mathbf{v}}_{2n} + \mathbf{e}\rVert_{\ell_2}$. 
    We now estimate
    \begin{align}\label{eq:LowerBoundOMPVector}
    \begin{split}
        \lVert \mathbf{A}\mathbf{p} - \mathbf{y}\rVert_{\ell_2} &= \lVert \mathbf{A}\mathbf{v} + \mathbf{e} - \mathbf{A}\mathbf{p}\rVert_{\ell_2}
        = \lVert \mathbf{A}\mathbf{v}_{2n} + \mathbf{A}\overline{\mathbf{v}}_{2n} + \mathbf{e} - \mathbf{A}\mathbf{p}\rVert_{\ell_2} = \lVert \mathbf{A}(\mathbf{v}_{2n} - \mathbf{p}) + \mathbf{A}\overline{\mathbf{v}}_{2n} + \mathbf{e}\rVert_{\ell_2} \\
        &\geq \lVert \mathbf{A}(\mathbf{v}_{2n} - \mathbf{p})\rVert_{\ell_2} - \lVert \mathbf{A}\overline{\mathbf{v}}_{2n} + \mathbf{e}\rVert_{\ell_2} \geq \sqrt{1-\delta_{26n}} \lVert \mathbf{v}_{2n} - \mathbf{p}\rVert_{\ell_2} - \lVert \mathbf{A}\overline{\mathbf{v}}_{2n} + \mathbf{e}\rVert_{\ell_2}
    \end{split}
    \end{align}
    The combination of \eqref{eq:LowerBoundOMPVector} and the estimate after \eqref{eq:UpperBoundOMPVector} yields
    \begin{equation}
        \lVert \mathbf{v}_{2n} - \mathbf{p}\rVert_{\ell_2} \leq \frac{50}{\sqrt{1-\delta_{26n}}} \lVert \mathbf{A}\overline{\mathbf{v}}_{2n} + \mathbf{e}\rVert_{\ell_2},
    \end{equation}
    which, applied to \cite[Lemma 6.23]{FoRa13} with $\xi = 0$, gives the desired result.
\end{proof}

\subsection{Recovery of functions}
In the following, we consider Wiener-type spaces, which are closely related to Barron spaces and have recently been studied, e.g., in \cite{DPW24, MSU25, NNS22}.
They have good embeddings into some Sobolev and Besov spaces with dominating mixed smoothness, see \cite[Sec.\ 5]{MSU25}, and are useful for approximation of \(L_q\) spaces, see Remark \ref{rem:FuncClass}.
Now, we establish the main results of this section, which involve the Wiener-type spaces $\mathcal{A}_{\theta}$ with $\theta\in (0,\infty)$ given by
\begin{equation}\label{eq:DefWiener}
    \mathcal{A}_\theta \coloneqq \mathcal{A}_\theta(\mathcal{B}) \coloneqq\bigl\{f\in L_1(\mu)\colon \; \lVert f\rVert_{\mathcal{A}_\theta}\coloneqq \lVert [f]\rVert_{\ell_\theta}<\infty\bigr\}
\end{equation}

with respect to a basis $\mathcal{B}\subset L_2(\mu)$. 
To unify the notation, we henceforth use the generic operator $H_\mathbf{X}$, which can be either the \rlasso \ decoder \smash{$R_{\mathbf{X}}^{\lambda}$} with $\lambda$ as in \Cref{prop:rLasso_for_NSP}, or the \omp \ decoder \smash{$P_{\mathbf{X}}^{24n}$} from \eqref{eq:OMP_Decoder}.
The corresponding vector $\mathbf{r}(\mathbf{y})$ or $\mathbf{p}^{24n}(\mathbf{y})$ is denoted by $\mathbf{h}(\mathbf{y})$.
Below, we combine \Cref{thm:RIP_implies_NSP}, \Cref{prop:rLasso_for_NSP} and \Cref{prop:OMP_for_RIP} with the general RIP result \Cref{thm:RIP} for the matrix $\mathbf{A}$ in \eqref{eq:matrix} induced by a bounded orthonormal system (BOS).

\begin{prop}\label{prop:Recovery_vector_errors}
Let $\mathcal{B} \subset L_2(\mu)$ be a BOS and \smash{$B \coloneqq \sup_{j \in I}\lVert b_j\rVert_{L_\infty}$}.
For given $m\in\mathbb{N}$ and $J\subset I$ finite, we define the (random) matrix $\mathbf{A}\in\mathbb{C}^{m\times\lvert J\rvert}$ as in \eqref{eq:matrix} with \smash{$\mathbf{X}\overset{iid}{\sim} \mu$}. 
Moreover, for given $n$ and $\mathbf{y}$, $\mathbf{h}(\mathbf{y})$ denotes the reconstruction vector of \rlasso \ with $\lambda(n)$ according to \Cref{prop:rLasso_for_NSP}, or \omp \ after at most $24n$ iterations. 
There exist universal constants $\alpha, \beta>0$ such that for all $n \in \mathbb{N}$, $\gamma \in (0,1)$ and 
\begin{equation}\label{eq:NumberSamples}
	m \geq \alpha B^2 n \bigl({\log^2} \bigl(B^2n\bigr) \log\lvert J\rvert + \log\bigl(B^2n\bigr)\log\bigl(\gamma^{-1}\log\bigl(B^2n\bigr)\bigr)\bigr),
\end{equation}
we have with probability at least $1-\gamma$
\begin{equation}\label{eq:Recovery_lq_error}
    \lVert \mathbf{v} - \mathbf{h}(\mathbf{y}) \rVert_{\ell_q} \leq \beta\bigl(n^{\sfrac{1}{q}-1}\sigma_n(\mathbf{v})_{\ell_1} + n^{\sfrac{1}{q}-\sfrac{1}{2}}  \lVert\mathbf{A}\mathbf{v}-\mathbf{y}\rVert_{\ell_2}\bigr)
\end{equation}
for all $q\in[1,2]$, all $\mathbf{v} \in \mathbb{C}^{\lvert J\rvert}$ and all ${\mathbf{y}} \in \mathbb{C}^m$.
\end{prop}

\begin{proof}
First, we discuss the \rlasso \ case.
By \Cref{thm:RIP}, there exist $\alpha>0$ such that, for $\gamma \in (0,1)$, $m$, $\vert J\vert$ and $n$ such that \eqref{eq:NumberSamples} holds, the matrix $\mathbf{A}$ has the RIP of order $2n$ with $\delta_{2n}<\sfrac{1}{3}$ with probability at least $1 - \gamma$.
For the instances where the RIP holds, \Cref{thm:RIP_implies_NSP} implies that $\mathbf{A}\in\operatorname{NSP}(n,2,\lVert\cdot\rVert_{\ell_2},\varrho,\tau)$, where $\varrho,\tau$ formally depend on $\delta_{2n}$.
However, since $\delta_{2n}<1/3$ we obtain universal constants $\varrho,\tau>0$ from \eqref{eq:NSP_proof} in its proof.
Then, \Cref{prop:rLasso_for_NSP} implies that there exists $\beta>0$ depending on $\varrho, \tau$ such that \eqref{eq:Recovery_lq_error} holds. 
For \omp, we again use \Cref{thm:RIP} to see that there is an $\alpha>0$ such that $\mathbf{A}$ has the RIP with $\delta_{26n}<\sfrac{1}{13}$ with probability at least $1 - \gamma$.
Now, the recovery bound \eqref{eq:Recovery_lq_error} holds by \Cref{prop:OMP_for_RIP}.
\end{proof}

Next, we translate this result on the level of coefficients into a result for functions $f\in \mathcal A_1$.

\begin{theorem}\label{thm:Error_Recovery_Ops}
There exist universal constants $\alpha, C_B>0$ such that for $n \in \mathbb{N}$ and $J\subset I$ finite with $n<\lvert J\rvert$ the following holds.
If we have at least as many sample points \smash{$\mathbf{X}=\{x^1,\ldots,x^m\} \overset{iid}{\sim} \mu$} as required in \eqref{eq:NumberSamples}, then with probability at least $1-\gamma$ we have
\begin{gather}\label{eq:ErrorEstimateRIP}
	\lVert f - H_{\mathbf{X}}(f)\rVert_{L_q} \leq C_B n^{\sfrac{1}{2}-\sfrac{1}{q}}\big(n^{-\sfrac{1}{2}}\sigma_n(f;\mathcal{B})_{\mathcal{A}_1} + \lVert f-P_Jf\rVert_{L_\infty}\big)
\end{gather}
for any $2\leq q\leq \infty$ and $f \in \mathcal{A}_1$.
\end{theorem}

\begin{proof}
First, we consider $q=\infty$.
Applying \eqref{eq:Recovery_lq_error} from \Cref{prop:Recovery_vector_errors} with $\mathbf{v} = P_Jf$ and $\mathbf{y} = m^{-\sfrac{1}{2}}f(\mathbf{X})$ yields that
\begin{align}\label{eq:Est_PN_term}
\begin{split}
    \lVert P_Jf - H_{\mathbf{X}}(f) \rVert_{L_\infty} &\leq \sum_{j\in J} \lVert b_j\rVert_{L_\infty} \lvert[P_Jf]_j - [H_{\mathbf{X}}(f)]_j \rvert \leq B \lVert[P_Jf] - [H_{\mathbf{X}} (f)]\rVert_{\ell_1}\\
    &\leq B\beta\big(\sigma_n([P_Jf])_{\ell_1} + n^{\sfrac{1}{2}} m^{-\sfrac{1}{2}} \lVert P_Jf(\mathbf{X}) - f(\mathbf{X})\rVert_{\ell_2}\big)\\
    &\leq B\beta\big(\sigma_n(P_Jf)_{\mathcal{A}_1} + n^{\sfrac{1}{2}} \lVert P_Jf - f\rVert_{L_\infty}\big)
\end{split}
\end{align}
holds with the stated probability. Using \eqref{eq:Est_PN_term} and $ \sigma_n([P_Jf])_{\ell_1} \leq \sigma_n([f])_{\ell_1} = \sigma_n(f)_{\mathcal{A}_1}$, we further obtain
\begin{align}
\begin{split}
    \lVert f - H_{\mathbf{X}}(f)\rVert_{L_\infty} & \leq \lVert f - P_Jf \rVert_{L_\infty} + \lVert P_Jf - H_{\mathbf{X}}(f) \rVert_{L_\infty}\\
    &\leq B \beta \sigma_n([P_Jf])_{\ell_1} + (B\beta n^{\sfrac{1}{2}}+1)\lVert f - P_Jf\rVert_{L_\infty}\\
    &\leq C_{\infty} n^{\sfrac{1}{2}} \bigl(n^{-\sfrac{1}{2}}\sigma_n(f)_{\mathcal{A}_1} + \lVert f-P_Jf\rVert_{L_\infty} \bigr).
\end{split}
\end{align}
This concludes the case $q =\infty$.

For $q = 2$, we note that $\lVert\cdot\rVert_{L_2}\leq \lVert\cdot\rVert_{L_\infty}$.
Thus, the Parseval identity and \eqref{eq:Recovery_lq_error} from \Cref{prop:Recovery_vector_errors} imply
\begin{align}
\begin{split}
    \lVert f - H_{\mathbf{X}}(f)\rVert_{L_2} & \leq \lVert P_Jf -  H_{\mathbf{X}}(f)\rVert_{L_2} + \lVert f -P_Jf \rVert_{L_\infty}\\
    & = \lVert [P_Jf] - [H_{\mathbf{X}}(f)]\rVert_{\ell_2} + \lVert f - P_Jf \rVert_{L_\infty}\\
    &\leq \beta n^{-\sfrac{1}{2}}\sigma_n([P_Jf])_{\ell_1} + (\beta+1) \lVert f - P_Jf\rVert_{L_\infty}\\
    & \leq C_2 \big( n^{-\sfrac{1}{2}} \sigma_n(f)_{\mathcal{A}_1} + \lVert f-P_Jf\rVert_{L_\infty} \big).
\end{split}
\end{align}

For $2<q<\infty$, H\"older's inequality implies
\begin{gather}
\lVert f - H_{\mathbf{X}}(f)\rVert_{L_q} \leq \lVert f - H_{\mathbf{X}}(f)\rVert_{L_2}^{1-\theta}\lVert f - H_{\mathbf{X}}(f)\rVert_{L_\infty}^\theta,
\end{gather}
where $\theta$ has to be chosen such that $\sfrac{1}{q}= \sfrac{(1-\theta)}{2}+ \sfrac{\theta}{\infty}$, which yields $\theta = 1 - \sfrac{2}{q}$.
Now, the desired estimate holds with probability at least $1-\gamma$.
\end{proof}

\begin{remark}[Function classes]\label{rem:FuncClass}
\hfill
\begin{enumerate}
\item If $f = P_J f$, the second term in \eqref{eq:ErrorEstimateRIP} vanishes.
Thus, the decoder satisfies \smash{$\lVert f - H_{\mathbf{X}}(f)\rVert_{L_q} \leq C_{B} n^{-\sfrac{1}{q}}\sigma_n(f;\mathcal{B})_{\mathcal{A}_1}$} with high probability if $m \gtrsim B^2n \log^2(B^2 n) \log\lvert J\rvert$, which is a simplification of the sampling complexity bound \eqref{eq:NumberSamples} by estimating $\log\bigl(\gamma^{-1}\log\bigl(B^2n\bigr)\bigr)$ in terms of $\log(B^2  n)$ at the cost of a constant depending on \(\gamma\).

\item For the general case, the truncation error does not vanish.
However, when recovering all functions from the unit ball of a given and fixed function class $\mathbf{F}$, we directly obtain from \Cref{thm:Error_Recovery_Ops} that
\begin{equation}\label{eq:FS}
    \sup\limits_{\|f\|_{\mathbf{F}}\leq 1} \|f-H_{\mathbf{X}}f\|_{L_q} \leq 
  C_B n^{\sfrac{1}{2}-\sfrac{1}{q}}\big(n^{-\sfrac{1}{2}}\sigma_n(\mathbf{F};\mathcal{B})_{\mathcal{A}_1} + \|I-P_J\|_{\mathbf{F}\to L_\infty}\big).
\end{equation}
A similar estimate with $q=2$ has been established in \cite{Kri23} for Basis Pursuit Denoising.
According to \eqref{eq:FS}, it suffices to have a mild control of the worst-case truncation error, say
\begin{equation}
    \|I-P_J\|_{\mathbf{F}\to L_\infty} = \sup_{\|f\|_{\mathbf{F}}\leq 1} \lVert f-P_Jf\rVert_{L_\infty} = \mathcal O(|J|^{-\kappa})
\end{equation}
for some $\kappa > 0$.
Since \(m\) depends only logarithmically on $|J|$, we may control the second summand in \eqref{eq:ErrorEstimateRIP} by the order of the first one at the price of using  now $m\gtrsim   n\log(n)^3$ many samples. This requires an at most polynomial decay of $\sigma_n(\mathbf{F};\mathcal{B})_{\mathcal{A}_1}$.
Inserting the bounds from \cite{MSU25} shows that our decoders outperform any linear decoder such as (weighted) least-squares for weighted Wiener and certain Sobolev-Besov spaces with mixed smoothness, see for instance \cite[Thm.\ 6.1 and Cor.\ 6.8]{MSU25}.
\end{enumerate}
\end{remark}

\subsection{General orthonormal bases}\label{Legendre}

We now consider an orthonormal basis $\mathcal B$ with atoms $b_j$ that are not uniformly bounded in $L_\infty$.
\begin{definition}\label{def:ChangeOfMeasure}
Let $\mathcal{B}\subset L_2(\mu)$ be orthonormal and assume that $\tilde{\psi}(x)\coloneqq\sup_{j\in I}\lvert b_j(x) \rvert \in L_2(\mu)$.
Then, we can define
\begin{equation}
\psi \coloneqq \frac{\tilde{\psi}}{\lVert\tilde{\psi}\rVert_{L_2(\mu)}} \quad \text{and} \quad \varphi^2 \coloneqq \frac{1+\psi^2}{2}
\end{equation}
and a probability measure $\tilde \mu$ on $\Omega$ with $\mathrm d \tilde \mu = \varphi^2 \mathrm d \mu$.
By setting $\tilde{b}_j \coloneqq \frac{b_j}{\varphi}$, we get an orthonormal basis $\tilde{\mathcal{B}}=\{\tilde{b}_j\colon \; j\in I\}$ for $L_2(\varphi^2\mathrm{d}\mu)$.
Moreover, $\tilde{\mathcal{B}}$ is uniformly bounded, such that we can define
\begin{gather}
\tilde{B}\coloneqq \sup\limits_{j\in I}\lVert \tilde{b}_j\rVert_{L_\infty(\mu)} \leq \sqrt{2} \lVert \tilde{\psi}\rVert_{L_2(\mu)}.
\end{gather}
\end{definition}
The assumption in \Cref{def:ChangeOfMeasure} is not trivial and may fail for certain bases, e.g, wavelets.
On the other hand, it holds for Jacobi polynomials, see \cite{RaWa12}, and the special case of Legendre polynomials is discussed in \Cref{rem:Legendre}. The method described in the definition was first described in \cite[Equation 6.7]{BSU23} and is a variation of the preconditioning trick which assures that $\varphi^2$ is bounded from below by $\sfrac{1}{2}$. The preconditioning trick itself was first described in \cite{RaWa12} for Legendre polynomials and in \cite{HamDoo15} for the general case. For further details, an overview is given in \cite[Sec.\ 5.5]{AdBrWe22}.

In principle, we can apply \Cref{thm:Error_Recovery_Ops} for the new basis $\tilde{\mathcal{B}}$.
However, this would lead to estimates with the modified measure $\tilde \mu$.
If we want to formulate a result similar to \Cref{thm:Error_Recovery_Ops} based on the original measure $\mu$, we have to modify the two decoders \rlasso \ and \omp.
For this, we again fix a finite index set $J$, and define the matrices
\begin{gather}\label{eq:mod_matrix}
	\tilde{\mathbf{A}} \coloneqq m^{-\sfrac{1}{2}}(\tilde{b}_j(x^l))_{1 \leq l \leq m, j \in J} \in \mathbb{C}^{m\times \lvert J\rvert} \quad \text{and} \quad \tilde{\mathbf{B}} \coloneqq \bigl(\tilde{\mathbf{A}}_j \cdot \lVert \tilde{\mathbf{A}}_j\rVert_{\ell_2}^{-1}\bigr)_{j\in J}
\end{gather}
and the vector
\begin{gather}
\tilde{\mathbf{y}} \coloneqq m^{-\sfrac{1}{2}} \frac{f}{\varphi}(\mathbf{X}) = m^{-\sfrac{1}{2}}\biggl(\frac{f(x^l)}{\varphi(x^l)}\biggr)_{l=1}^m \in \mathbb{C}^m.
\end{gather}

\begin{definition}[Modified \rlasso]\label{def:mod_square_root_lasso}
For $\lambda>0$, we define $\tilde{R}_{\mathbf{X}}^\lambda\colon C(\Omega) \to V_J$ by
\begin{equation}\label{eq:mod_Lasso_Decoder}
	\tilde{R}_{\mathbf{X}}^\lambda(f) = \sum_{j \in J} \tilde{r}_j(\tilde{\mathbf{y}}) b_j,
\end{equation}
where $\tilde{\mathbf{r}}(\tilde{\mathbf{y}}) = [\tilde{R}_{\mathbf{X}}^\lambda(f)] \in \mathbb{C}^{\lvert J\rvert}$ is a (fixed) solution of the \rlasso \ minimization problem
\begin{equation}\label{eq:mod_MinimizationProblem}
	\min_{\mathbf{z} \in \mathbb{C}^{\lvert J\rvert}} \lVert\mathbf{z}\rVert_{\ell_1} + \lambda \lVert\tilde{\mathbf{A}} \mathbf{z} - \tilde{\mathbf{y}}\rVert_{\ell_2}
\end{equation}
with the matrix $\tilde{\mathbf{A}}$ from \eqref{eq:mod_matrix}.
Note that we implicitly assume that there is a deterministic selection rule in the case of multiple solutions.
\end{definition}

\begin{definition}[Modified \omp]\label{def:mod_OMP} 	
For $n\in \mathbb{N}$, we define $\tilde{P}_{\mathbf{X}}^n\colon C(\Omega) \to V_J$ by
\begin{gather}\label{eq:mod_OMP_Decoder}
    \tilde{P}_{\mathbf{X}}^{n}(f) = \sum_{j \in J} \tilde{p}_j^n(\tilde{\mathbf{y}}) b_j,
\end{gather}
where $\tilde{\mathbf{p}}^n(\tilde{\mathbf{y}}) =[\tilde{P}_{\mathbf{X}}^n(f)]\in\mathbb{C}^{\lvert J\rvert}$ is recursively defined through
\begin{align}
    S^{k+1}(\tilde{\mathbf{y}}) & = S^k \cup \operatorname{argmax} \bigl\{\lvert(\tilde{\mathbf{B}}^\ast (\tilde{\mathbf{y}} - \tilde{\mathbf{B}}\tilde{\mathbf{p}}^k))_j\rvert\colon \;  1 \leq j \leq \lvert J\rvert\bigr\},\label{eq:mod_OMP_iteration1}\\
    \tilde{\mathbf{p}}^{k+1}(\tilde{\mathbf{y}}) & = \operatorname{argmin}\big\{\lVert \tilde{\mathbf{y}} - \tilde{\mathbf{B}}\mathbf{p}\rVert_{\ell_2}\colon \; \supp(\mathbf{ p}) \subset S^{k+1}\big\}\label{eq:mod_OMP_iteration2}
\end{align}
with $S^0(\mathbf{y}) \coloneqq \emptyset$ and $\tilde{\mathbf{p}}^0(\tilde{\mathbf{y}}) \coloneqq \mathbf{0}$. If the $\argmax$ is not a singleton, we select the smallest index $j$.
\end{definition}

\begin{remark}
The decoders use the modified basis $\tilde{\mathcal{B}}$ only to obtain the coefficients $\tilde{\mathbf{r}}(\tilde{\mathbf{y}})$ and $\tilde{\mathbf{p}}^n(\tilde{\mathbf{y}})$. The reconstruction in \eqref{eq:mod_Lasso_Decoder} and \eqref{eq:mod_OMP_Decoder} is still based on the original basis $\mathcal{B}$.
\end{remark}

Analogously to our previous notation, we now use the generic operator $\tilde{H}_{\mathbf{X}}$ to denote either the modified \rlasso \ $\tilde{R}_{\mathbf{X}}^{\lambda}$ with $\lambda$ as in \Cref{prop:rLasso_for_NSP}, or the modified \omp \ output $\tilde{P}_{\mathbf{X}}^{24n}$.

\begin{theorem}\label{thm:Recovery_General_Bases}
There exists a universal constant $\alpha>0$ such that for all $n \in \mathbb{N}$, $\gamma \in (0,1)$ and $J\subset I$ finite with $n<\lvert J\rvert$ the following holds.
If we draw
\begin{equation}
	m \geq \alpha \tilde{B}^2 n \bigl({\log^2} (\tilde{B}^2n) \log\lvert J\rvert + \log\bigl(\tilde{B}^2n\bigr)\log\bigl(\gamma^{-1}\log(\tilde{B}^2n)\bigr)\bigr)
\end{equation}
sample points $\mathbf{X}=\{x^1,\ldots,x^m\} \overset{iid}{\sim} \varphi^2\,d\mu$, then with probability at least $1-\gamma$ we have, with $\beta$ from \Cref{prop:Recovery_vector_errors},
\begin{gather}
	\lVert f - \tilde{H}_{\mathbf{X}}(f)\rVert_{L_2(\mu)} \leq (\sqrt{2}\beta+1)\big(n^{-\sfrac{1}{2}}\sigma_n(f;\mathcal{B})_{\mathcal{A}_1} + \lVert f-P_Jf\rVert_{L_\infty(\mu)}\big)
\end{gather}
for any $f \in \mathcal{A}_1$.
\end{theorem}

\begin{proof}
We fix $f\in \mathcal{A}_1$ and note that $\lVert\cdot\rVert_{L_2(\mu)}\leq \lVert\cdot\rVert_{L_\infty(\mu)}$.
We then apply in turn the Parseval identity and \Cref{prop:Recovery_vector_errors} to obtain
\begin{align}
\begin{split}
    \lVert f - \tilde{H}_{\mathbf{X}}(f)\rVert_{L_2(\mu)} & \leq \lVert P_Jf - \tilde{H}_{\mathbf{X}}(f)\rVert_{L_2(\mu)} + \lVert f -P_Jf \rVert_{L_\infty(\mu)}\\
    & = \lVert [P_Jf] - [\tilde{H}_{\mathbf{X}}(f)]\rVert_{\ell_2} + \lVert f - P_Jf \rVert_{L_\infty(\mu)}\\
    &\leq \beta n^{-\sfrac{1}{2}}\sigma_n(P_Jf, \mathcal{B})_{\mathcal{A}_1} + \beta \lVert \sfrac{1}{\varphi} \cdot (f - P_Jf)\rVert_{L_\infty(\mu)} + \lVert f - P_Jf\rVert_{L_\infty(\mu)}\\
    &\leq \beta n^{-\sfrac{1}{2}}\sigma_n(P_Jf, \mathcal{B})_{\mathcal{A}_1} + (\sqrt{2}\beta+1) \lVert f - P_Jf\rVert_{L_\infty(\mu)}\\
    & \leq (\sqrt{2}\beta+1) \bigl( n^{-\sfrac{1}{2}} \sigma_n(f, \mathcal{B})_{\mathcal{A}_1} + \lVert f-P_Jf\rVert_{L_\infty(\mu)} \bigr).
\end{split}
\end{align}
This is due to the fact that $\varphi \geq \frac{1}{\sqrt{2}}$ and concludes the proof.
\end{proof}

\begin{remark}\label{rem:Legendre}
    The Legendre polynomials $L_k\colon [-1, 1] \to \mathbb{R}$, $k\in\mathbb{N}_0$, with
    \begin{equation}
        x \mapsto \frac{\sqrt{2k+1}}{2^kk!} \cdot \frac{\mathrm{d}^k}{\mathrm{d}x^k}\bigl(x^2-1\bigr)^k
    \end{equation}
    are an orthonormal basis for the normalized Lebesgue measure $\mathrm{d}\mu=\sfrac{1}{2}\mathrm{d}x$.
    This basis is not uniformly bounded.
    However, due to \cite{AntHol81}, one can show that the atoms $L_k$ are bounded by the $L_2(\mu)$-function
    \begin{equation}
        \tilde{\psi}_L(x) \coloneqq
        \begin{cases}
            \sfrac{2}{\sqrt{\pi}}\, (1-x^2)^{-\sfrac{1}{4}} & x\in (-1, 1)\\
            \infty & x\in\{-1,1\}
        \end{cases}.
    \end{equation}
    The version stated here can be found in \cite[Eq.\ 5.49]{AdBrWe22}.
    Since \smash{$\lVert\tilde{\psi}_L\rVert_{L_2}^2=2$}, we can apply \Cref{thm:Recovery_General_Bases} to obtain error bounds for \rlasso \ or \omp \ with Legendre polynomials.
\end{remark}

\section{Minimal number of samples -- instance optimality}\label{sec:InstanceOptimality}
In this section, we will derive lower bounds on the number of required samples for any decoder with a certain error decay.
Our results are based on the notion of instance optimality for vectors, which we briefly introduce below.
The proofs are along the lines of \cite[Sec.\ 11]{FoRa13} and included in \hyperref[sec:ProofsInstance]{Appendix \ref*{sec:ProofsInstance}} for convenience.

\subsection{Instance optimality notions for vectors}\label{sec:InstanceOptVec}

The concept of instance optimality was first introduced by Cohen, Dahmen and DeVore in \cite{CDD09}.

\begin{definition}
Let $p,q\in[1,\infty]$. We call a pair $(\mathbf{A},\Delta)$ of a measurement matrix $\mathbf{A}\in\C^{m\times N}$ and a reconstruction map $\Delta \colon \C^m\to\C^N$ $(q,p)$--$n$--instance-optimal with constant $C_N>0$ if it satisfies for all $\mathbf{z}\in\C^N$ that
\begin{gather}\label{eq:InstOpt}
\lVert\mathbf{z}-\Delta(\mathbf{A}\mathbf{z})\rVert_{\ell_q}\leq C_Nn^{\sfrac{1}{q}-\sfrac{1}{p}}\sigma_n(\mathbf{z})_{\ell_p},
\end{gather}
denoted by $(\mathbf{A},\Delta)\in\operatorname{IO}(q,p,n,C_N)$. If $p=q$, then we say that $(\mathbf{A},\Delta)$ is $p$--$n$--instance-optimal with constant $C_N$ and write $(\mathbf{A},\Delta)\in\operatorname{IO}(p,n,C_N)$.
\end{definition}

The following statement is a slight generalization of \cite[Thm.\ 11.4]{FoRa13}.
It can be obtained by tracking the involved constants and generalizing from the $\ell_1$- to the $\ell_p$-case.

\begin{prop}\label{prop:IO_to_best_n_term}
Let $q\geq p\geq 1$ and $\mathbf{A}\in\C^{m\times N}$.
If there exists $\Delta\colon \C^m \to\C^N$ such that $(\mathbf{A},\Delta)\in\operatorname{IO}(q,p,n,C_N)$, then we have
\begin{gather}\label{eqn:IOqNorm}
\lVert\mathbf{z}\rVert_{\ell_q}\leq C_Nn^{\sfrac{1}{q}-\sfrac{1}{p}}\sigma_{2n}(\mathbf{z})_{\ell_p}
\end{gather}
for all $\mathbf{z}\in\ker(\mathbf{A})$. Conversely, if \eqref{eqn:IOqNorm} holds for all $\mathbf{z}\in\ker(\mathbf{A})$, then there exists $\Delta\colon \C^m \to\C^N$ such that $(\mathbf{A},\Delta) \in \operatorname{IO}(q,p,n,2C_N)$.
\end{prop}
In the definition of instance optimality, we can actually reduce the $q$ as desired. 
\begin{lemma}[{\cite[Thm.\ 11.8]{FoRa13}}] \label{thm:IOPreservation}
Let $q\geq q'\geq p\geq 1$ and $(\mathbf{A},\Delta)\in\operatorname{IO}(q,p,n,C_N)$.
Then there exists $\Delta'\colon \C^m \to\C^N$ such that $(\mathbf{A},\Delta')\in\operatorname{IO}(q',p,n,6C_N+2)$.
\end{lemma}

The following statement is a slight modification of \cite[Thm.\ 11.6]{FoRa13}, which is used to derive the main result of this section thereafter.

\begin{prop}\label{thm:LowerBoundSingleNorm}
Let $\mathbf{A}\in\C^{m\times N}$ and $\Delta\colon \C^m \to\C^N$.
If $(\mathbf{A},\Delta) \in\operatorname{IO}(p,n,C_N)$ for some $p\geq 1$ and $n \in \N$, then
\begin{gather}
m\geq \frac{n\log(N/4n)}{4\log(2C_N+3)}.
\end{gather}
\end{prop}

\begin{theorem}\label{instance_vector}
Let $\mathbf{A}\in\C^{m\times N}$ and $\Delta\colon \C^m \to\C^N$. Suppose that $(\mathbf{A},\Delta)\in\operatorname{IO}(q,p,n,C_N)$ for $q\geq p\geq 1$. Then
\begin{gather}
    m\geq\frac{n\log(N/4n)}{4\log(12C_N+7)}.
\end{gather}
\end{theorem}

\begin{proof}
\Cref{thm:IOPreservation} for $q'=p$ ensures the existence of a reconstruction map $\Delta'\colon \C^m \to\C^N$ such that $(\mathbf{A},\Delta')\in\operatorname{IO}(p,n,6C_N+2)$.
Now, \Cref{thm:LowerBoundSingleNorm} yields the desired result.
\end{proof}

\subsection{Minimal number of samples for sparse recovery}
With the tools from \Cref{sec:InstanceOptVec} at hand, we can establish a lower bound on the number $m$ of sampling points in $\mathbf{X}$ if a reconstruction operator shall achieve a certain recovery accuracy.
\begin{theorem}\label{instance_B1B1}Let $\mathcal{B} \subset L_2(\mu)$ be an ONB and $T_{\mathbf{X}} \colon \mathcal{A}_1 \to L_2$ a recovery operator of the form $T_{\mathbf{X}}(f) = T \circ f(\mathbf{X})$ with $T\colon \C^m \to L_2$, and let $\lVert\cdot\rVert_V$ be some norm.
If
\begin{equation}\label{eq:Instance_Optimality_A1A1}
\lVert f - T_{\mathbf{X}}(f) \rVert_{\mathcal{A}_1} \leq C \big(\sigma_n(f)_{\mathcal{A}_1}+\lVert f - P_J(f)\rVert_V\big)
\end{equation}
or
\begin{equation}\label{eq:Instance_Optimality_L2A1}
\lVert f - T_{\mathbf{X}}(f) \rVert_{L_2} \leq C\big(n^{-\sfrac{1}{2}}\sigma_n(f)_{\mathcal{A}_1} + \lVert f - P_J(f)\rVert_V\big)
\end{equation}
holds for all $f\in\mathcal{A}_1$, a constant $C>0$ and $n\in\mathbb{N}$, then there exists $C'>0$ such that
\begin{equation}\label{eq:LowerBoundM}
m \geq C'n\log(\lvert J\rvert/4n).
\end{equation}
\end{theorem}

\begin{proof}

We want to employ \Cref{thm:LowerBoundSingleNorm} and \Cref{instance_vector}, for which we consider the measurement matrix $\mathbf{A}\in\C^{m\times\lvert J\rvert}$ with $A_{l,j}= b_j(x^l)$.
Further, we define the recovery operator $(\Delta(\mathbf{c}))_j=[T(\mathbf{c})]_j$ for $j\in J$ and $\mathbf{c}\in\mathbb{C}^m$.
Now, for any \smash{$\mathbf{z}\in\C^{\lvert J\rvert}$} we define
\begin{gather}
f_\mathbf{z} = \sum\limits_{j\in J} z_jb_j\in V_J.
\end{gather}
By construction, we have $\mathbf{Az} = f_{\mathbf z}(\mathbf X)$ and $\lVert f_{\mathbf{z}} - P_J(f_{\mathbf{z}}) \rVert_{V} = 0$.
We will show that we need sufficiently many samples even to recover these functions.
In the following, we distinguish the two cases.
If $T_{\mathbf{X}}$ satisfies \eqref{eq:Instance_Optimality_A1A1}, we apply the definition of the space $\mathcal{A}_1$ in \eqref{eq:DefWiener} to obtain
\begin{equation}\label{vec2fct_A1A1}
\lVert \mathbf{z} - \Delta(\mathbf{A}\mathbf{z}) \rVert_{\ell_1} = \lVert f_{\mathbf{z}} - P_J \circ T_{\mathbf{X}}(f_{\mathbf{z}})\rVert_{\mathcal{A}_1} \leq \lVert f_{\mathbf{z}} - T_{\mathbf{X}}(f_{\mathbf{z}})\rVert_{\mathcal{A}_1} \leq C\sigma_n(f_{\mathbf{z}})_{\mathcal{A}_1} = C\sigma_n(\mathbf{z})_{\ell_1},
\end{equation}
and see that $(\mathbf{A},\Delta)\in\operatorname{IO}(1,n,C)$.
Hence, applying \Cref{thm:LowerBoundSingleNorm} yields \eqref{eq:LowerBoundM}.
If instead $\eqref{eq:Instance_Optimality_L2A1}$ holds, we apply the Parseval theorem to get
\begin{equation}\label{vec2fct}
\lVert \mathbf{z} - \Delta(\mathbf{A}\mathbf{z}) \rVert_{\ell_2} = \lVert f_{\mathbf{z}} - P_J \circ T_{\mathbf{X}}(f_{\mathbf{z}})\rVert_{L_2} \leq \lVert f_{\mathbf{z}} - T_{\mathbf{X}}(f_{\mathbf{z}}) \rVert_{L_2} \leq Cn^{-\sfrac{1}{2}}\sigma_n(f_{\mathbf{z}})_{\mathcal{A}_1} = Cn^{-\sfrac{1}{2}}\sigma_n(\mathbf{z})_{\ell_1},
\end{equation}
that is $(\mathbf{A},\Delta)\in\operatorname{IO}(2,1,n,C)$.
Thus, applying \Cref{instance_vector} concludes the proof.
\end{proof}

\section{The multivariate trigonometric systems as a special case}\label{sec:Fourier}

In what follows, we specify our results for the multivariate trigonometric system
\begin{gather}
\mathcal B= \mathcal{T}^d = \{\exp(2\pi\mathrm{i}\mathbf{k}\cdot)\colon \; \mathbf{k}\in \mathbb{Z}^d\}
\end{gather}
defined on the torus $\Omega= [0, 1]^d$.
By $Q_M\coloneqq [-M,M]^d\cap\mathbb{Z}^d$ we denote the $d$-dimensional discrete $M$-cube, and by $\mathcal{T}(Q_M) = \spann_\C(\exp(2\pi\mathrm{i}\mathbf{k}\cdot)\colon \; \mathbf{k}\in Q_M)$ the associated trigonometric polynomials.
For $D \in \mathbb{N}$, we consider the equidistant grid
\begin{gather}
    G\coloneqq G(D,d) \coloneqq \Big\{ \frac{\mathbf{n}}{2D}\colon \; \mathbf{n} \in \{0, \ldots, 2D\}^d \Big\} \subset \mathbb{T}^d.
\end{gather}
If we randomly draw $m$ iid points $\mathbf{X} = \{\mathbf{x}^l\colon \; l\in[m]\}$ from the discrete uniform distribution $\mu_G = \lvert G\rvert^{-1} \sum_{\mathbf{x}\in G} \delta_\mathbf{x}$, we can consider the associated random Fourier matrix
\begin{equation}\label{eq:FourierMat}
    \mathbf{A} \coloneqq\big(m^{-\sfrac{1}{2}} \exp(2\pi\mathrm{i}\mathbf{k}\cdot \mathbf{x}^{l})\big)_{l\in[m],\mathbf{k}\in Q_D}.
\end{equation} 

The following theorem has been proven in \cite{MPU25_1} for \rlasso \ and can be proven analogously for \omp. It relies on the RIP bound by Haviv and Regev in \cite{HaRe17}. 
Note that we have the $L_\infty$-error as well as the best approximation on the right-hand side.
Its proof relies on a specifically designed multivariate de la Vall{\'e}e Poussin type operator, see \cite{JUV23}.

\begin{theorem}\label{Thm:trigo}
There exist universal constants $\alpha, C, \gamma>0$ such that for all $M,n \in \mathbb{N}$ and $D \coloneqq (2d+1)M$ the following holds.
If we draw
\begin{gather}\label{eq:NumberSamplesFourier}
     m\geq \alpha d n \log^2(n+1)  \log(D+1)
\end{gather}
samples $\mathbf{X}=\{\mathbf{x}^1,\ldots, \mathbf{x}^m\} \overset{iid}{\sim} \mu_G$, then with probability at least $1-(2D+1)^{-d\gamma \log (n+1)}$ the decoder $H_{\mathbf{X}}$, which is either \rlasso \ from \Cref{def:square_root_lasso} or \omp \ from \Cref{def:OMP} on the search space $V_{Q_D}$, satisfies
\begin{gather}\label{eq:rLasso_Fourier}
	\lVert f - H_{\mathbf{X}}(f)\rVert_{L_q} \leq  Cn^{\sfrac{1}{2}-\sfrac{1}{q}} \bigl(\sigma_n(f;\mathcal{T}^d)_{L_\infty} + \inf_{g \in\mathcal{T}(Q_M)}\lVert f - g\rVert_{L_\infty} \bigr),
\end{gather}
for any $f \in C(\mathbb{T}^d)$ and $2 \leq q \leq \infty$.
\end{theorem}
Note that the size of the search space enters only logarithmically here. Let us compare this with the general result \Cref{thm:Error_Recovery_Ops}.
Since $\mathcal B = \mathcal{T}^d$ is bounded by one in $L_\infty$, the number of required samples \eqref{eq:NumberSamples} and \eqref{eq:NumberSamplesFourier} behave similarly.
Compared to the recovery estimate \eqref{eq:ErrorEstimateRIP}, the estimate \eqref{eq:rLasso_Fourier} involves \smash{$\sigma_n(f;\mathcal{T}^d)_{L_\infty}$} instead of \smash{$n^{-\sfrac{1}{2}}\sigma_n(f;\mathcal{T}^d)_{\mathcal{A}_1}$}.
This has the striking advantage that one may directly insert classical known bounds on $\sigma_n(f;\mathcal{T}^d)_{L_\infty}$, see \cite{Belinskii1989, MSU25, TeUl22} and \cite[Sec.\ 7]{DTU18} for an overview.

\subsection[Minimal number of samples for sparse recovery on the \textit{d}-torus]{Minimal number of samples for sparse recovery on the $d$-torus}
With the next theorem, we may get rid of $\sigma_{2n}(f;\mathcal{T}^d)_{L_\infty(\mathbb{T}^d)}$ and use the results from \Cref{instance_B1B1}.

\begin{theorem}[{\cite[Thm.\ 4.5]{MSU24_1}}]\label{thm:best_n_term_trig_pol}
    For a trigonometric polynomial $f\in\mathcal{T}(Q_M)$ on $\mathbb T^d$ we have
    \begin{gather}
        \sigma_{4n}\bigl(f,\mathcal{T}^d\bigr)_{L_\infty}\leq C \sqrt{d}n^{-\sfrac{1}{2}}\log(N)^{\sfrac{1}{2}}\lVert f\rVert_{\mathcal{A}_1}
    \end{gather}
    with $N=(2M+1)^d$ and a constant $C<460$.
\end{theorem}

The following result shows that $m = \mathcal{O}(n)$ samples do not suffice to obtain \eqref{eq:rLasso_Fourier} for all~$f$.

\begin{theorem}
For every recovery map $T_{\mathbf{X}} \colon C(\mathbb{T}^d) \to L_2(\mathbb{T}^d)$ of the form $T_{\mathbf{X}}(f) = T \circ f(\mathbf{X})$ with $T\colon \C^m \to L_2(\mathbb{T}^d)$ that satisfies the condition
\begin{equation}\label{eqn:InsOptLinf}
\lVert f - T_{\mathbf{X}}(f) \rVert_{L_2} \leq C\bigl(\sigma_{2n}(f;\mathcal{T}^d)_{L_\infty} + \inf_{g \in\mathcal{T}(Q_M)}\lVert f - g\rVert_{V} \bigr)
\end{equation}
for every $f\in C(\mathbb{T}^d)$ and $n \in \N$, the number of required samples is
\begin{gather}
    m \geq n\, \frac{\log(N/4n)}{4\log(C'd\log(N))},
\end{gather}
where $C'$ is an absolute constant and $N=(2M+1)^d$.
\end{theorem}


\begin{proof}
We follow the proof of \Cref{instance_B1B1} up to \eqref{vec2fct}.
Using the assumption \eqref{eqn:InsOptLinf} then yields
\begin{equation}
\lVert \mathbf{z} - \Delta(\mathbf{A}\mathbf{z}) \rVert_{\ell_2} = \lVert f_{\mathbf{z}} - P_{Q_M}\circ T_{\mathbf{X}}(f_{\mathbf{z}}) \rVert_{L_2} \leq C\sigma_{2n}(f_{\mathbf{z}})_{L_\infty}\,,    
\end{equation}
where $P_{Q_M}$ denotes the projection onto $\mathcal{T}(Q_M)$.
Next, we choose an $n$-sparse trigonometric polynomial $s$ such that $\lVert f_\mathbf{z} - s \rVert_{\mathcal{A}_1} = \sigma_n(f_\mathbf{z})_{\mathcal{A}_1}$, which is possible since $f$ is a trigonometric polynomial, i.e.\ we only consider finitely many indices. We obviously see that $\sigma_{2n}(f_{\mathbf{z}})_{L_\infty} \leq \sigma_{n}(f_{\mathbf{z}}-s)_{L_\infty}$.
Since $f_\mathbf{z}-s \in \mathcal{T}(Q_M)$ we get, after employing \Cref{thm:best_n_term_trig_pol},
\begin{equation}
\sigma_{2n}(f_\mathbf{z})_{L_\infty} \leq \sigma_n(f_\mathbf{z}-s)_{L_\infty} \leq \tilde{C}\sqrt{d} n^{-\sfrac{1}{2}} \log(N)^{\sfrac{1}{2}} \lVert f_\mathbf{z} -s \rVert_{\mathcal{A}_1}
\end{equation}
where $\tilde{C}\geq 1$.
Now, since $\lVert f_\mathbf{z} -s \rVert_{\mathcal{A}_1} = \sigma_n(f_\mathbf{z})_{\mathcal{A}_1} \leq 
\sigma_n(\mathbf{z})_{\ell_1}$, this implies that $(\mathbf{A},\Delta)\in\operatorname{IO}(2,1,n,C\cdot\tilde{C}\sqrt{d}\log(N)^{\sfrac{1}{2}})$, such that the claim follows by applying \Cref{instance_vector}.
\end{proof}

\begin{remark}\label{rem:IO} The result in \Cref{Thm:trigo} has an interesting counterpart for optimal approximation from linear subspaces using samples. In \cite[Thm.\ 17]{KPUU24}, which represents a local version of \cite{LiTe22}, it is shown that for any $n$-dimensional subspace $V_n \subset C(\Omega)$ we can find $m = 4n$ points $\mathbf{X} = \{x^1,\ldots,x^m\}$ such that the classical least-squares decoder
\begin{equation}\label{eq:lsqrop}
    A(\mathbf{X},V_n)f\coloneqq \argmin\limits_{g\in V_n}\Big(\frac{1}{m}\sum\limits_{i=1}^m |f(x_i)-g(x_i)|^2\Big)^{\sfrac{1}{2}}
\end{equation}
satisfies 
\begin{equation}
    \|f-A(\mathbf{X},V_n)f\|_{L_2} \leq 84\inf\limits_{g \in V_n}\|f-g\|_{{L_\infty}}
\end{equation}
for all $f\in C(\Omega)$. 
This indicates a scaling of the number $m$ of samples as $m=\mathcal O(n)$.
However, these are adapted to $V_n$.
On the other hand, \eqref{eqn:InsOptLinf} shows that a statement of the form
\begin{equation}\label{eq:IO}
    \|f-T_{\mathbf{X}}f\|_{L_2} \leq C\Big(\inf\limits_{|J|\leq n}\inf\limits_{c_k\in \C, k\in J} \Big\|f-\sum\limits_{k\in J} c_kb_k\Big\|_{L_\infty} + \inf\limits_{g \in \mathcal{T}(Q_M)}\|f-g\|_{L_{\infty}}\Big)
\end{equation}
for all $f\in C(\mathbb{T}^d)$ requires
\begin{equation}
    m\geq cn\frac{d\log(M/4n)}{\log(d\log(M))}
\end{equation}
many samples no matter which point set $\mathbf{X}$ and decoder $T_{\mathbf{X}}$ we use.
The second summand in \eqref{eq:IO} cannot be dropped since a corresponding statement cannot hold true with finitely many samples.
This shows that a scaling of $m=\mathcal O(n)$ is not possible if we switch from linear subspaces to $n$-term sums.  
\end{remark}

%% file: sections/numerics.tex
\section{Numerical experiments}\label{sec:Numerics}

Now, we verify our theoretical observations with numerical simulations for the Fourier system from \Cref{sec:Fourier} as well as the Chebyshev system.
First, we briefly describe the key components of our implementations.
The \texttt{Python} code is available on GitHub\footnote{\href{https://github.com/Zeppo1994/SparseRecovery}{https://github.com/Zeppo1994/SparseRecovery}}.

\subsection{Implementation details}
\subsubsection{Scalable matrix-vector multiplications}
Since the size $\lvert J\rvert$ of the search space commonly scales with the dimension $d$ (grid-based discretizations even scale exponentially), we require a memory-efficient and scalable implementation of the matrix-vector products with the sampling matrix  $\mathbf{A} \in \mathbb{C}^{m \times \lvert J\rvert}$ introduced in \eqref{eq:matrix}, and its adjoint.
In certain instances, such as for the Fourier system with dimension $d \leq 4$, see \cite{B1995,PPST18}, fast algorithms can beat the naive computational complexity $\mathcal{O}(m\lvert J\rvert)$.
Instead of restricting ourselves to such special cases, we use a parallel and matrix-free implementation for generic $\mathbf{A}$ based on \texttt{PyKeops}  \cite{ChaFeyGla2021}.
To this end, we require (preferably simple) closed-form formulas for the atoms $b_j$, this is also the main reason why we did not implement the Legendre system.
Compared to a classical matrix-based implementation, we have to recompute the entries of $\mathbf{A}$ on the fly so as to gain memory efficiency.
Our implementation also enables straightforward GPU acceleration using \texttt{PyTorch}.
Since we typically have $m \ll \lvert J\rvert$, the performance gap of our approach compared to fast methods is often relatively small.
One downside is that the single precision limit of GPUs is reached quite fast.
Thus, we do not include computations where the relative error gets below \num{1e-4}.

\subsubsection{Primal-dual method for \rlasso}

The \rlasso \ decoder from \Cref{def:square_root_lasso} involves a non-smooth convex optimization problem whose efficient numerical solution has been investigated in several works, see \cite{AdcCol2025,BeCheWa11,MayMelKra2024}.
Following Adcock et al.\ (\cite{AdcCol2025}), we deploy the primal-dual hybrid gradient method (PDHGM) \cite{ChaPoc2011} with restarts to minimize the \rlasso \ objective
\begin{equation}\label{eq:Sq_Lasso_Numerics}
		\mathcal J(\mathbf z) \coloneqq \lVert \mathbf{A} \mathbf{z} - {\mathbf{y}}\rVert_{\ell_2} + \tfrac{1}{\lambda} \lVert\mathbf{z}\rVert_{\ell_1}.
\end{equation}
To this end, we perform the decomposition $\mathcal J(\mathbf z) = F(\mathbf{A} \mathbf{z}) + G(\mathbf{z})$ with $F(\mathbf{x}) = \lVert \mathbf x - \mathbf y \rVert_{\ell_2}$ and $G(\mathbf{z}) = \frac{1}{\lambda}\lVert \mathbf{z} \rVert_{\ell_1}$.
The resulting PDHGM iterations are given in \Cref{alg:PDHGM}.
This algorithm relies on the proximal operators of $G$ and the convex conjugate $F^{\ast}$, which are the softshrinkage
\begin{equation}
    \mathrm{prox}_{\tau G}(x) = S_{\sfrac{\tau}{\lambda}}(x) = \begin{cases}
        \mathrm{sign}(x) (\vert x \vert - \sfrac{\tau}{\lambda}) & \text{if } \vert x \vert > \sfrac{\tau}{\lambda},\\
        0 & \text{else},
    \end{cases}
\end{equation}
(which is applied entry-wise) and the rescaling  $\mathrm{prox}_{\sigma F^{\ast}}(\mathbf q) =
(\mathbf q - \sigma \mathbf y)/\max(1, \lVert \mathbf q - \sigma \mathbf y \rVert_{\ell_2})$.
The averaged iterates \smash{$\mathbf z_\mathrm{avg}^{(K)}$} and \smash{$\mathbf q_\mathrm{avg}^{(K)}$} converge if the step sizes satisfy $\tau \sigma \leq 1/ \lVert \mathbf{A} \rVert_{2}^2$, see \cite[Thm.\ 1]{ChaPoc2011}.
If unknown, the spectral norm $\lVert \mathbf{A} \rVert_{2}$ can be estimated with power iterations.

Next, we discuss the optimal choice of the step sizes $\tau$ and $\sigma$.
Let $\mathbf Z \subset \mathbb C^{\lvert J \rvert}$ denote the set of minimizers of \eqref{eq:Sq_Lasso_Numerics}.
For the initial iterate $\mathbf{z}^{(0)}$, we denote the (a priori unknown) distance to $\mathbf Z$ by $\delta = \mathrm{dist}(\mathbf{z}^{(0)},\mathbf Z)$.
Given a target optimality gap $\mathcal J(\mathbf{z}^{(K)}) - \mathcal J(\mathbf Z)\leq \varepsilon$, we initialize $\mathbf q^{(0)}= \mathbf 0$ and choose $\tau = \delta/\lVert \mathbf{A} \rVert_2$, $\sigma = 1/(\delta \lVert \mathbf{A} \rVert_2)$.
Then, performing at most \smash{$K(\varepsilon)=\frac{2\delta}{\varepsilon}\lVert \mathbf{A} \rVert_2$} iterations leads to the desired accuracy, see \cite[Prop.\ 4.12]{AdcCol2025}.
Even if we would know $\delta$, the asymptotically expected number of iterations is $K(\varepsilon) \sim \varepsilon^{-1}$.
This makes obtaining highly accurate solutions costly and an improved approach is necessary.

\begin{algorithm}[ht]
\caption{Primal-Dual Hybrid Gradient Method (PDHGM) for \rlasso}\label{alg:PDHGM}
\textbf{Input}: Matrix $\mathbf A$, values $\mathbf{y}$, initializations $\mathbf{z}^{(0)}$ and $\mathbf q^{(0)}$, step sizes $\tau$ and $\sigma$, iterations $K$\\
\textbf{Parameters}: Regularization strength $\lambda$
\begin{algorithmic}[1]
\Function{PD-rLASSO}{$\mathbf{z}^{(0)}, \mathbf q^{(0)}, \tau, \sigma, K$}
\State \smash{$\mathbf z_\mathrm{avg}^{(0)} = \mathbf z^{(0)}$, $\mathbf q_\mathrm{avg}^{(0)} = \mathbf q^{(0)}$}
\ForAll{$k=0,\ldots,K-1$}
\State \smash{$\mathbf z^{(k+1)} = S_{\tau/\lambda}(\mathbf z^{(k)} - \tau \mathbf A^{\top} \mathbf q^{(k)})$}
\State \smash{$\tilde{\mathbf q} = \mathbf q^{(k)} + \sigma \mathbf A(2\mathbf z^{(k+1)} - \mathbf z^{(k)}) - \sigma \mathbf y$}
\State \smash{$\mathbf q^{(k+1)} = \tilde{\mathbf q}/\max(1, \lVert \tilde{\mathbf q} \rVert_{\ell_2})$}
\State \smash{$\mathbf z_\mathrm{avg}^{(k+1)} = (k \mathbf z_\mathrm{avg}^{(k)} + \mathbf z^{(k+1)})/ (k+1)$}
\State \smash{$\mathbf q_\mathrm{avg}^{(k+1)} = (k \mathbf q_\mathrm{avg}^{(k)} + \mathbf q^{(k+1)})/ (k+1)$}
\EndFor
\State \textbf{return} \smash{$\mathbf z_\mathrm{avg}^{(K)}$, $\mathbf q_\mathrm{avg}^{(K)}$}
\EndFunction
\end{algorithmic}
\end{algorithm}

As demonstrated in \cite{AdcCol2025}, restarts can significantly improve the asymptotic convergence behavior if an approximate sharpness condition of the form
\begin{equation}\label{eq:ApproxSharp}
    \mathrm{dist}(\mathbf z, \mathbf Z) \leq \Bigl(\frac{\mathcal J(\mathbf z) - \mathcal J(\mathbf Z) + \eta}{\alpha}\Bigr)^{\sfrac{1}{\beta}} \qquad \text{for all } \mathbf z \in \mathbb C^{\lvert J\rvert} \text{ such that } \mathcal J(\mathbf z) \leq \mathcal J\bigl(\mathbf z^{(0)}\bigr)
\end{equation}
is satisfied with constants $\eta \geq 0$, $\beta\geq 1$ and $\alpha>0$.
For $\eta=0$, the condition \eqref{eq:ApproxSharp} coincides with the notion of sharpness introduced in \cite{RouAsp2020}.
It is weaker than other typical conditions (such as strong convexity of $\mathcal J$) that lead to improved rates.
Similar to the Kurdyka-{\L}ojasiewicz condition, \eqref{eq:ApproxSharp} implies that $\mathcal J$ cannot be arbitrarily flat around its minima $\mathbf Z$ (with threshold $\eta$).
As shown in \cite[Subsec.\ 5.3]{AdcCol2025}, the \rlasso \ objective $\mathcal J$ in \eqref{eq:Sq_Lasso_Numerics} satisfies the condition \eqref{eq:ApproxSharp} provided that $\mathbf A$ has the robust NSP.
The latter is established in \Cref{thm:RIP_implies_NSP} assuming the RIP.

\begin{algorithm}[ht]
\caption{Restarted PDHGM for \rlasso}\label{alg:rsPDHGM}
\textbf{Input}: Matrix $\mathbf A$, values $\mathbf{y}$, initialization $\mathbf{z}^{(0)}$, number of PDHGM calls $R$\\
\textbf{Parameters}: Regularization strength $\lambda$, sharpness constants $\alpha$ and $\beta$
\begin{algorithmic}[1]
\Function{RS-PD-rLASSO}{$\mathbf{z}^{(0)}, R$}
\State \smash{$\tau=\sigma=1/ \lVert \mathbf{A} \rVert_2$}
\State \smash{$\varepsilon^{(0)} = \mathcal J(\mathbf z^{(0)})$}; \smash{$\mathbf q^{(0)}= \mathbf 0$}
\ForAll{$k=0,\ldots,R-1$}
\State \smash{$\delta^{(k+1)}  = (2\varepsilon^{(k)}/ \alpha)^{1 / \beta}$}
\State \smash{$\varepsilon^{(k+1)} = \varepsilon^{(k)} / \mathbf e$}.
\State \smash{$K^{(k+1)}= \lceil 2\delta^{(k+1)}\lVert A \rVert_2 / \varepsilon^{(k+1)}\rceil$}
\State \smash{$\mathbf z^{(k+1)}, \mathbf q^{(k+1)} = \Call{PD-rLASSO}{\mathbf{z}^{(k)}, \mathbf q^{(k)}, \tau \delta^{(k+1)}, \sigma / \delta^{(k+1)}, K^{(k+1)}}$}
\EndFor
\EndFunction
\State \Return \smash{$\mathbf z^{(R)}$}\phantomsection\label{rLassoRestarts:Return}
\end{algorithmic}
\end{algorithm}

Given $\alpha$ and $\beta$, we can deploy the restart scheme in \Cref{alg:rsPDHGM}.
For $\eta \lesssim \varepsilon$, this leads to an expected number of iterations $K(\varepsilon) \sim \lVert \mathbf{A} \rVert_2\log(\sfrac{1}{\varepsilon})\alpha^{-1}$ for $\beta=1$, and $K(\varepsilon) \sim \lVert \mathbf{A} \rVert_2 \varepsilon^{\sfrac{1}{\beta}-1}\alpha^{-\sfrac{1}{\beta}} $ for $\beta>1$, see \cite[Tab.\ 1]{AdcCol2025}.
Hence, the approximate sharpness condition \eqref{eq:ApproxSharp} enables improved convergence rates until the desired accuracy $\varepsilon$ is in the order of $\eta$.
If the sharpness constants $\alpha$, $\beta$ in \eqref{eq:ApproxSharp} are unknown, we can deploy a line search, see \cite[Alg.\ 2]{AdcCol2025}, which leads to an additional logarithmic factor in the expected number of iterations for $\beta>1$.
To accelerate our implementation, we terminate the PDHGM in \hyperref[rLassoRestarts:Return]{Line \ref*{rLassoRestarts:Return}} of \Cref{alg:rsPDHGM} as soon as the (estimated) optimality gap is less than $\varepsilon$.
This is justified by the fact that $K(\varepsilon)$ was chosen accordingly.
 Since the cost of each iteration in \Cref{alg:rsPDHGM} is dominated by the application of $\mathbf A$ and $\mathbf{A}^{\top}$, we get the overall complexity $\mathcal O(m \lvert J\rvert\varepsilon^{\sfrac{1}{\beta}-1}\lVert \mathbf{A} \rVert_2\alpha^{-\sfrac{1}{\beta}}$).
In our experiments, we set $\beta=2$ for the Fourier system and $\beta=3$ for the Chebyshev system.
The parameter $\alpha$ is set to $\alpha = \Vert \mathbf A \Vert_2$ for the Fourier system and $\alpha = 0.3 \Vert \mathbf A \Vert_2$ for the Chebyshev system.
Moreover, we use $11$ restarts.

\subsubsection{Orthogonal Matching Pursuit}

The Orthogonal Matching Pursuit (\omp) from \Cref{def:OMP} is a greedy algorithm, see \cite{KuRa08,Zha11}, of which a brief pseudocode is given in \Cref{alg:OMP}.
For every update, we add one entry to the set \(S^{(k)}\) of nonzero coefficients, thereby successively increasing the support of the approximation \(\mathbf{z}^{(k)}\).
In \hyperref[OMP:ApproxUpdate]{Line \ref*{OMP:ApproxUpdate}} of the \omp \ algorithm, we have to compute the best least-squares fit with the current support.
To solve the (linear) optimality condition \smash{$\mathbf{C}_k^{*} \mathbf{C}_k \mathbf{z}^{(k)} = \mathbf{C}_k^{*} \mathbf{y}$}, where \smash{$\mathbf{C}_k \coloneqq \mathbf{B}_{\,:,S^{(k)}}$} (and $\mathbf{B}$ is the matrix $\mathbf{A}$ with $\ell_2$-normalized columns), we can iteratively update the Cholesky decomposition of \smash{$\mathbf{C}_k^{*} \mathbf{C}_k$} as detailed in \cite{StuChr2012}.
This approach leads to an improved complexity of $\mathcal O(\lvert J \rvert m + mk + k^2)$ per iteration compared to the naive implementation with complexity $\mathcal O(\lvert J \rvert m + m k^2 + k^3)$.
Since we expect $\mathbf{C}_k$ to be well-conditioned, solving the normal equation is numerically stable.
In principle, we can also solve the least-squares problem with an iterative method.
However, we found this to be efficient only for sufficiently large systems, namely if the iteration count $K$ is large.

Since every step of \Cref{alg:OMP} adds precisely one support entry, the \omp \ reconstruction \(\mathbf{z}^{(K)}\) is always \(K\)-sparse.
While this is a useful property when proving theoretical results, this means that the number of iterations $K$ is strongly related to the support of the reconstruction.
Thus, to ensure a fair comparison, we aim to choose $K$ in the same order as the sparsity of the \rlasso \ reconstruction, which often results in long runtimes for our experiments. Therefore for particularly large experiments we sometimes capped the number of iterations which leads to an early plateau for the error provided by OMP.

\begin{algorithm}[H]
\caption{\omp}\label{alg:OMP}
\textbf{Input}: Matrix \(\mathbf{A}\), values $\mathbf{y}$, number of iterations $K$
\begin{algorithmic}[1]
\Function{OMP}{$K$}
\State \smash{$S^{(0)} = \emptyset$}
\State \smash{$\mathbf{B} = \mathbf{A}\mathbf{D}$, with $\mathbf{D} = \operatorname{diag}\bigl(\lVert\mathbf{A}_j\rVert_{\ell_2}^{-1}\bigr)_{j\in J}$}
\ForAll{$k=0,\ldots,K-1$}
\State \smash{$ j^{(k+1)} = \operatorname{argmax}_{1 \leq j \leq \lvert J\rvert}\{\lvert({\mathbf{B}}^\ast ({\mathbf{B}}{\mathbf{z}}^{(k)} - \mathbf{y}))_j\rvert\}$}
\State \smash{\(S^{(k+1)} = S^{(k)} \cup \{ j^{(k+1)}\}\)}
\State \smash{${\mathbf{z}}^{(k+1)} = \operatorname{argmin}\{\lVert {\mathbf{B}}\mathbf{z} - {\mathbf{y}} \rVert_{\ell_2}\colon \; \mathbf{z}\in \mathbb{C}^{\lvert J\rvert}, \operatorname{supp}(\mathbf{z}) \subset S^{(k+1)}\}$}\phantomsection\label{OMP:ApproxUpdate}
\EndFor
\EndFunction
\State \textbf{return} \smash{$\mathbf z^{(K)}$}
\end{algorithmic}
\end{algorithm}

\subsubsection{Compressive Sampling Matching Pursuit}

To improve the runtime, we turn to a modification of \omp \ called Compressive Sampling Matching Pursuit (\cosamp), see \cite{NeeTr09}, a brief pseudocode of which can be found in \Cref{alg:CoSaMP}.
This algorithm also aims to find the most efficient basis elements in every step, but, unlike \omp \ in \Cref{alg:OMP}, it exchanges up to $n$ basis elements per iteration.
To this end, it first adds up to $2n$ indices in each iteration.
Then, it computes the best least-squares fit with this support set and only keeps the $n$ indices with the highest magnitude.
To solve the least-squares problem, we use the iterative LSQR method, and initialize it with the relevant part of the previous solution.
Hence, we expect that the solver converges in less than 20 iterations after a few
outer iterations of \cosamp, in particular if $\mathbf A$ is well-conditioned.
Moreover, $K = \mathcal O(\log(\Vert \mathbf c \Vert_2 / \eta))$ outer iterations suffice to ensure $\Vert \mathbf z^{(K)} -  \mathbf c \Vert_2\leq \max \{\eta, n^{-\sfrac{1}{2}} \sigma_n(\mathbf c)_{\ell_1}\}$, see \cite[Thm.\ A]{NeeTr09}.
Taking \cite[Thm.\ 6.28]{FoRa13} into account, there is a strong indication that an equivalent to \Cref{thm:Error_Recovery_Ops} holds.
For this, we denote with $C_{K,\mathbf{X}}(f)$ the output of \Cref{alg:CoSaMP} after $K$ iterations.
\begin{theorem}\label{thm:Error_Recovery_Cosamp}
There exist universal constants $\alpha, C_B>0$ such that for $n \in \mathbb{N}$ and $J\subset I$ finite with $n<\lvert J\rvert$ the following holds.
If we draw at least as many sample points \smash{$\mathbf{X}=\{x^1,\ldots,x^m\} \overset{iid}{\sim} \mu$} as required in \eqref{eq:NumberSamples}, then with high probability  we have for all $f \in \mathcal{A}_1$
\begin{equation}\label{eq:ErrorCOSaMP}
	\lVert f - C_{K,\mathbf{X}}(f)\rVert_{L_q} \leq C_B n^{\sfrac{1}{2}-\sfrac{1}{q}} \bigl(n^{-\sfrac{1}{2}} \sigma_n(f;\mathcal{B})_{\mathcal{A}_1} + \lVert f-P_Jf\rVert_{L_\infty} + \varrho^K\lVert f\rVert_{L_2}\bigr)
\end{equation}
for any $2\leq q\leq \infty$.  Here $0<\varrho<1$ is a constant depending only on the RIP constant $\delta_{8n}<0.47$.
\end{theorem}
The additional third term in the error bound \eqref{eq:ErrorCOSaMP} is getting exponentially small if the number of iterates $K$ is getting large.
In order to obtain the desired polynomial decay in $n$, we would need roughly $K = \mathcal{O}(\log n)$ many iterations, which shows that \cosamp\ is much faster than \omp~ also in theory.
In fact, instead of $\mathcal{O}(|J|mn)$ flops as for \omp, \cosamp \ only requires $\mathcal{O}(|J|m\log n)$ flops, see also \cite[Tab.\ 2]{NeeTr09} and the associated discussion.

\begin{algorithm}[H]
\caption{\cosamp}\label{alg:CoSaMP}
\textbf{Input}: Matrix \(\mathbf{A}\), values $\mathbf{y}$, number of iterations $K$, sparsity \(n\)
\begin{algorithmic}[1]
\Function{CoSaMP}{$K$}
\State $ \mathbf{z}^{(0)} = \mathbf 0$
\State \smash{$\mathbf{B} = \mathbf{A}\mathbf{D}$, with $\mathbf{D} = \operatorname{diag}\bigl(\lVert\mathbf{A}_j\rVert_{\ell_2}^{-1}\bigr)_{j\in J}$}
\ForAll{$k=0,\ldots,K-1$}
\State \smash{\(S^{(k+1)} = \operatorname{supp}(\mathbf{z}^{(k)}) \cup \operatorname{argmax}\{\lvert({\mathbf{B}}^\ast ({\mathbf{B}}{\mathbf{z}}^{(k)} - \mathbf{y}))_L\rvert :  L \subset J, \lvert L \rvert = 2n\}\)}
\State \smash{${\mathbf{u}}^{(k+1)} = \operatorname{argmin}\{\lVert {\mathbf{B}}\mathbf{z} - {\mathbf{y}} \rVert_{\ell_2}\colon \; \mathbf{z}\in \mathbb{C}^{\lvert J\rvert}, \operatorname{supp}(\mathbf{z}) \subset S^{(k+1)}\}$}
\State \smash{\( S^{(k+1)}  = \operatorname{argmax} \{ \lvert \mathbf{u}_L^{(k+1)} \rvert \colon \; L \subset J, \lvert L \rvert = n \} \)}
\State \smash{\( \mathbf{z}^{(k+1)} = \mathbf{u}^{(k+1)}_{S^{(k+1)} } \)}
\EndFor
\EndFunction
\State \textbf{return} \smash{$\mathbf z^{(K)}$}
\end{algorithmic}
\end{algorithm}

\subsection{Function recovery}
We perform recovery experiments for several functions from the literature.
Even for moderate dimensions $d$, we already suffer from the problem that a grid-based search space $J$ grows exponentially with the dimension $d$.
Hence, we restrict ourselves to smaller search spaces such as a hyperbolic cross \cite{DTU18}.
Clearly, other restrictions are also possible.
Alternatively, to avoid the search space $V_J$, one can also optimize over the atoms directly, see \cite{HerNeu2025} for details.
All plots in this section drop the $m \asymp n \log^3(n)$ relation from \Cref{prop:Recovery_vector_errors}.

\subsubsection{The multivariate Fourier system}
First, we discuss approximation for the Fourier system, which was treated as special case in \Cref{sec:Fourier}.
Given $m$ uniformly distributed sampling points $\mathbf X \subset \mathbb{T}^d$, the values $\mathbf y = f(\mathbf X)$ at these positions, and a search space $J \subset \mathbb Z^{d}$, we aim to reconstruct the corresponding Fourier coefficients of $f$ on $J$ with both \rlasso \ and \omp. We use the hyperbolic cross $J_{s} \coloneqq \{\mathbf{k}\in\mathbb{Z}^d\colon \; \Pi(\mathbf{k})\leq s\}$ with the radius $s\in \mathbb{N}$, where
\begin{equation}
    \Pi(\mathbf{k}) \coloneqq \prod\limits_{i=1}^d (1+\lvert k_i\rvert),
\end{equation}
as search space, which has asymptotical cardinality $\vert J_s \vert \asymp s(\log s)^{d-1}$, see \cite{KSU15}.
For small $s>1$ there exist pre-asymptotic bounds such as $\vert J_s \vert \leq e^2s^{2+\log d}$, see \cite[Proof of Thm.\ 4.9]{KSU15}.
For \rlasso, we specify the regularization parameter $\lambda$ for each experiment individually.
Regarding \omp, we choose the number of steps as $K = \min(\vert J \vert, m, 20000)$.
Taking the minimum is justified as follows.
Since each step of \Cref{alg:OMP} adds one support point, it does not make sense to perform more than $\vert J \vert$ iterations.
Moreover, we do not want to end up with an underdetermined least squares problem.
Thus, we perform at most $m$ steps.

\begin{figure}[H]
    \centering
    \captionof{table}{Example 1 with \rlasso \ (top) and \omp \ (bottom): The frequencies form a hyperbolic cross and the samples are drawn uniformly from $[0,1]^5$.
    The errors are computed using the ground-truth Fourier coefficients. All values are up to sampling randomness.}\label{tab:func1_OMP}
    \tabcolsep=0.15cm
    \begin{tabular}{c|ccccccc|c}
        \toprule
        $\lvert J \rvert$ | $m$ & 1K & 5K & 10K & 50K & 100K & 300K & 600K & trunc.\ error\\
        \midrule
        4.5K  & \num{2.34e-1} & \num{8.23e-2} & \num{6.34e-2} & \num{4.57e-2} & \num{4.29e-2} & \num{3.89e-2} & \num{3.76e-2} &  \num{ 3.43e-2}\\
        12K  & \num{3.02e-1} & \num{7.36e-2} & \num{5.07e-2}  & \num{3.30e-2} & \num{2.94e-2} & \num{2.61e-2} & \num{2.47e-2} &  \num{2.12e-2}\\
        54K  & \num{4.17e-1} & \num{8.80e-2} & \num{4.37e-2}  & \num{1.28e-2} & \num{9.91e-3}  & \num{7.78e-3}  & \num{7.05e-3}  & \num{5.28e-3}\\
        104K & \num{4.12e-1} & \num{1.02e-1} & \num{5.03e-2} & \num{1.07e-2} & \num{7.17e-3}  & \num{5.05e-3}  & \num{4.42e-3}  &  \num{3.03e-3}\\
        311K & \num{5.32e-1} & \num{1.27e-1} & \num{6.36e-2} & \num{1.20e-2} & \num{5.70e-3}  & \num{2.74e-3}  & \num{2.18e-3}  &\num{1.16e-3}\\
        1M  & \num{6.71e-1} & \num{2.19e-1} & \num{8.03e-2}  & \num{1.61e-2} & \num{7.41e-3}  & \num{2.09e-3}  & \num{1.19e-3}  &\num{3.77e-4}\\
        \bottomrule
    \end{tabular}

    \vspace{3mm}
    \tabcolsep=0.15cm
    \begin{tabular}{c|ccccccc|c}
        \toprule
        $\lvert J \rvert$ | $m$ & 1K & 5K & 10K & 50K & 100K & 300K & 600K & trunc.\ error\\
        \midrule
        4.5K & \num{2.56e-1} & \num{7.86e-2} & \num{5.57e-2} & \num{4.19e-2} & \num{3.94e-2} & \num{3.72e-2} & \num{3.63e-2} & \num{ 3.43e-2}\\
        12K & \num{2.83e-1} & \num{7.83e-2} & \num{5.68e-2} & \num{2.99e-2} & \num{2.71e-2} & \num{2.44e-2} & \num{2.35e-2}& \num{2.12e-2}\\
        54K & \num{3.11e-1} & \num{8.09e-2} & \num{4.30e-2} & \num{1.17e-2} & \num{9.75e-3} & \num{8.95e-3} & \num{8.78e-3} & \num{5.28e-3} \\
        104K& \num{3.25e-1} & \num{8.31e-2} & \num{4.53e-2} & \num{1.02e-2} & \num{8.36e-3} & \num{7.81e-3} & \num{7.71e-3}& \num{3.03e-3}\\
        311K & \num{4.25e-1} & \num{9.39e-2} & \num{5.24e-2} & \num{1.06e-2} & \num{5.77e-3} & \num{5.03e-3} & \num{4.92e-3} & \num{1.16e-3}\\
        1M & \num{4.29e-1} & \num{9.86e-2} & \num{5.44e-2} & \num{1.24e-2}& \num{5.92e-3} & \num{2.97e-3} & \num{2.85e-3} &\num{3.77e-4}\\
        \bottomrule
    \end{tabular}

    \vspace{.2cm}

\centering
\captionof{figure}{Example 1: Approximation error depending on the number of samples for the \rlasso \ and \omp \ algorithm with the maximally sized $J$ ($\vert J \vert = 1M$).
    The plot is in log scale.}
    \label{fig:Fkt1rLASSO}

\vspace{2mm}

\begin{tikzpicture}
\begin{axis}[
    width=12cm,
    height=8cm,
    xlabel={Number of samples $m$ or terms $n$.},
    ylabel={$L_2$-Error},
    xmode=log,
    ymode=log,
    log basis x=10,
    log basis y=10,
    grid=both,
    grid style={dotted},
    mark=*,
]
\addplot coordinates {
    (1000,     0.7591481804847717 + 0.0002149017)
    (2000,     0.5952785015106201 + 0.0002149017)
    (4000,     0.4069393277168274 + 0.0002149017)
    (8000,     0.1459283679723739 + 0.0002149017)
    (16000,    0.0578027740120887 + 0.0002149017)
    (32000,    0.0284349527209997 + 0.0002149017)
    (64000,    0.0136513393372297 + 0.0002149017)
    (128000,   0.0064156851731240 + 0.0002149017)
    (256000,   0.0026716541033238 + 0.0002149017)
    (512000,   0.0010794417466968 + 0.0002149017)
    (1024000,  0.0003579959156922 + 0.0002149017)
    (2048000,  0.0001886745594674 + 0.0002149017)
};
\addlegendentry{\rlasso~approximant}

\addplot coordinates {
    (1000,     5.142367e-01)
    (2000,     2.513504e-01)
    (4000,     1.345149e-01)
    (8000,     7.668998e-02)
    (16000,    4.383527e-02)
    (32000,    2.254496e-02)
    (64000,    1.147291e-02)
    (128000,   5.010347e-03)
    (256000,   2.675015e-03)
    (512000,   2.470930e-03)
    (1024000,  2.407902e-03)
    (2048000,  2.379027e-03)
};
\addlegendentry{\omp~approximant}

\addplot[gray, mark=triangle*]coordinates {
    (1000,     0.062187109143 + 0.0000121849)
    (2000,     0.035435069352 + 0.0000121849)
    (4000,     0.020855875686 + 0.0000121849)
    (8000,     0.011940010823 + 0.0000121849)
    (16000,    0.006612705998 + 0.0000121849)
    (32000,    0.003610250074 + 0.0000121849)
    (64000,    0.001916919369 + 0.0000121849)
    (128000,   0.001008824562 + 0.0000121849)
    (256000,   0.000525453012 + 0.0000121849)
    (512000,   0.000268026110 + 0.0000121849)
    (1024000,  0.000134955146 + 0.0000121849)
    (2048000,  0.000066431465 + 0.0000121849)
};
\addlegendentry{best $n$-term}


\addplot[
    domain=1000:2048000,
    thick, no marks,
]
{1.2e-4 * x^(-1.5) * ln(x)^8};
\addlegendentry{$x^{-{\sfrac{3}{2}}} \log(x)^8$}
\end{axis}
\end{tikzpicture}
\end{figure}

\paragraph{Example 1, isotropic decay of coefficients.}
Following \cite[Sec.\ 10]{KamUllVol2021}, we consider the 1-dimensional function $g\colon \mathbb{T} \to \mathbb{R}$ given by
\begin{equation}\label{eq:Fourier_Approx_Basic_Func}
    g(x) = \frac{15}{4\sqrt{3}} \cdot 5^{\sfrac{3}{4}} \max\bigl\{0.2 - (x\bmod 1 -0.5)^2, 0\bigr\}.
\end{equation}
Its Fourier coefficients are $\hat{g}_0 = 3^{-\sfrac{1}{2}}\cdot 5^{\sfrac{1}{4}}$ and
\begin{equation}\label{eq:Fourier_Approx_Basic_Coeffs}
    \hat{g}_{k} = \frac{5^{\sfrac{5}{4}}\sqrt{3}}{8} \, (-1)^{k} \, \frac{\sqrt{5} \, \sin\left(2k\pi/\sqrt{5}\right) - 2 k \pi \cos\left(2k\pi/\sqrt{5}\right)}{\pi^3 k^3}, \qquad k\neq 0.
\end{equation}
To get an example in dimension $d$, we consider the tensor product $f\colon \mathbb{T}^d\rightarrow\R$ with $f(\mathbf{x}) = \prod_{i=1}^{d} g(x_i)$.
The corresponding Fourier coefficients are \smash{$\hat{f}_{\mathbf{k}} = \prod_{i=1}^{d} \hat{g}_{k_i}$} for any $\mathbf{k} \in \Z^d$.
With an elementary calculation, one can show that $\lVert f \rVert_{L_2} = 1$, which in particular means that the coefficients satisfy $\Vert \hat f \Vert_{\ell_2} = 1$.

Below, we investigate the theoretical quantities appearing in \Cref{thm:Error_Recovery_Ops}.
For the Fourier coefficients $\hat{f}_k$, we know that
\begin{equation}\label{eq:Fourierdecay}
    \lvert\hat{f}_{\mathbf{k}}\rvert \lesssim \Pi(\mathbf{k})^{-2}\eqqcolon a_{\mathbf{k}}\,,\quad \mathbf{k} \in \mathbb{Z}^d.
\end{equation}
Thus, using \cite[Formula (2.3.2)]{DTU18}, we directly get
\begin{equation}\label{eq:L2_trunc}
    \lVert f-P_{J_s}f\rVert_{L_2} \lesssim \biggl(\sum_{ \mathbf{k}\colon \Pi(\mathbf{k})>s} \Pi(\mathbf{k})^{-4}\biggr)^{\sfrac{1}{2}} \asymp s^{-\sfrac{3}{2}}(\log s)^{\sfrac{(d\!\!-\!\!1)}{2}} \asymp \lvert J_s \rvert^{-\sfrac{3}{2}}(\log \lvert J_s\rvert)^{2(d-1)}.
\end{equation}
Similarly, we obtain for the $L_\infty$-truncation error that
\begin{equation}
\label{eq:L_infty}
    \lVert f-P_{J_s}f\rVert_{L_\infty} \lesssim \sum\limits_{ \mathbf{k}\colon \Pi(\mathbf{k})>s} \Pi(\mathbf{k})^{-2} \asymp s^{-1}(\log s)^{d-1} \asymp \lvert J_s \rvert^{-1}(\log \lvert J_s\rvert)^{2(d-1)}.
\end{equation}
Further, we can obtain estimates for best $n$-term approximation of $f$ based on \eqref{eq:Fourierdecay}.
To this end, we non-increasingly rearrange the multi-indexed sequence $(a_{\mathbf{k}})_{\mathbf{k}}$ as
\begin{equation}\label{eq:rearr}
    a_n^{\ast} \asymp n^{-2}(\log n)^{2(d-1)},\quad n\in \mathbb{N},
\end{equation}
see \cite[Sects.\ 2 and 3]{KSU15}. Clearly, we have for any $q>0$ that $\sigma_n(f;\mathcal{B})_{\mathcal{A}_q} \lesssim (\sum_{k>n} \lvert a_k^{\ast}\rvert^q)^{\sfrac{1}{q}}$.
Using \eqref{eq:rearr}, we thus get
\begin{equation}\label{eq:bestmterm}
    \sigma_n(f;\mathcal{B})_{L_2} \lesssim \biggl(\sum\limits_{k>n} \lvert a_k^{\ast}\rvert^2 \biggr)^{\sfrac{1}{2}} \asymp n^{-\sfrac{3}{2}}(\log n)^{2(d-1)}
\end{equation}
and
\begin{equation}\label{eq:bestmterm2}
    n^{-\sfrac{1}{2}}\sigma_n(f;\mathcal{B})_{\mathcal{A}_1}\lesssim n^{-\sfrac{1}{2}} \sum\limits_{k>n} \lvert a_k^{\ast}\rvert \asymp n^{-\sfrac{3}{2}}(\log n)^{2(d-1)}.
\end{equation}

Hence, by choosing $s \asymp n^\kappa$ for some $\kappa>1$, we can ensure that the bound in \eqref{eq:L_infty} has the same order as \eqref{eq:bestmterm2}.
Then, \Cref{thm:Error_Recovery_Ops} implies that
\begin{equation}
	\lVert f - H_{\mathbf{X}}(f)\rVert_{L_2} \lesssim n^{-\sfrac{1}{2}}\sigma_n(f;\mathcal{B})_{\mathcal{A}_1} + \lVert f-P_{J_s}f\rVert_{L_\infty} \lesssim n^{-\sfrac{3}{2}}(\log n)^{2(d-1)},
\end{equation}
where $H_{\mathbf{X}}$ is either the \rlasso \ or the \omp \ decoder.
Both use $m = \lvert\mathbf{X}\rvert \asymp n(\log n)^3$ many samples.
The approximation results for $d=5$ and different choices of $m$ and sizes $\lvert J_s\rvert$  of the hyperbolic cross are given in \Cref{tab:func1_OMP}.
For \rlasso, we use the parameter \smash{$\lambda = m^{\sfrac{1}{2}}$}.
To compute the errors, we use the fact that
\begin{equation}\label{eq:ErrorSplitting}
    \lVert f - H_{\mathbf{X}}(f) \rVert^2_{L_2} = \sum_{\mathbf{k} \in J_s} \lvert [f]_{\mathbf{k}} - [H_{\mathbf{X}}(f)]_{\mathbf{k}} \rvert ^2 + \sum_{\mathbf{k} \notin J_s}  \lvert [f]_{\mathbf{k}}\rvert^2,
\end{equation}
namely that we do not require any approximation of the norm.
The second summand in \eqref{eq:ErrorSplitting} is provided in \Cref{tab:func1_OMP} as truncation error and acts as a theoretical lower bound.
As expected, the errors decay and eventually approach the truncation error if we provide more samples.
However, as we sample from the non-truncated function $f$ (which can be interpreted as noise), we cannot expect perfect recovery of the truncated coefficients.

A corresponding plot of the errors for the \rlasso \ and \omp \ with $\lvert J_s \rvert = 1M$ is given in \Cref{fig:Fkt1rLASSO}.
There, we also include the best $n$-term approximation of the sequence \eqref{eq:rearr} in $\ell_2$ together with the asymptotic bound in \eqref{eq:bestmterm} as benchmark.
The best $n$-term error matches the predicted asymptotic rates.
For more than $m=300k$ samples, \omp \ saturates due to the limited number of steps.
In this high accuracy regime, \rlasso \ appears to be the superior choice.

\begin{remark}[Least squares versus \rlasso \ and \omp]\label{rem:lsqr}
It is well-known that $f$ can be recovered using a (weighted) least-squares decoder with respect to a hyperbolic cross $J_s$ provided that the number of samples $m$ is chosen such that $m = \mathcal{O}(\lvert J_s \rvert \log(\lvert J_s \rvert))$, see \cite{CoMi16, KamUllVol2021,KrUl21}.
The error is then given by $\|f-P_{J_s}f\|_{L_2}$ in expectation for this individual $f$, see \cite[Sect.\ 6]{KamUllVol2021}, which, of course, is optimal with respect to the fixed basis $\mathcal{B}$.
In addition, the number of samples scales better than our bound \eqref{eq:NumberSamples} and can be even reduced to $O(|J_s|)$ using algorithmic subsampling \cite{BSU23}.
Thus, one might wonder why we care about non-linear decoders.
The answer is that the optimal search space for $L_2$-approximation is, of course, barely known in advance (at least in practice) and might differ significantly from $J_s$.
Our non-linear decoders do not require this knowledge and instead rely on the sparsity level $n$ and a control of the truncation error $\|f-P_{J}f\|_{L_\infty}$, see \Cref{rem:FuncClass}.

Sometimes, we can also infer reasonable search spaces from regularity assumptions.
The bounds \eqref{eq:L2_trunc} and \eqref{eq:L_infty} can be deduced from the fact that $f$ has mixed Besov regularity, i.e., \smash{$f \in \mathbf{B}^{2}_{1,\infty}(\mathbb{T}^d)$}, see \cite[Chapt.\ 3]{DTU18}.
Then, the hyperbolic cross $J_s$ is the optimal search space, see \cite[Thm.\ 4.2.6]{DTU18}.
Moreover, results in \cite{DKU23, KPUU24} show that we may recover $f$ using only $m = \mathcal{O}(\lvert J_s\rvert)$ many samples from a deterministically subsampled random point cloud with a weighted least squares decoder.
The $L_2$-error of this decoder is of order $\mathcal{O}(\lvert J\rvert^{-\sfrac{3}{2}}(\log \lvert J\rvert)^{2(d-1)})$, which is asymptotically optimal with respect to this class and slightly outperforms \rlasso \ and \omp \ in terms of the sample complexity.
\end{remark}

\hypertarget{Example2}{\paragraph{Example 2, anisotropic decay of coefficients.}}
We proceed to an example where the isotropic hyperbolic cross $J_{s}$ does not represent the optimal search space anymore.
Inspired by \cite[Sec.\ 5.3]{KamPotVol2021}, we consider the function $f \colon \mathbb{T}^7 \to \R$ with
\begin{equation}
    f(\mathbf x) = \frac{1}{\sqrt{2 + 2 C_2^3 C_4^4}} \biggl(\prod_{i \in \{1,3,7\}} N_2(x_i) + \prod_{i \in \{2,4,5,6\}} N_4(x_i)\biggr),
\end{equation}
where
\begin{equation}\label{eq:FourierBSpline}
    N_\ell(x) = C_\ell \sum_{k \in \Z} \mathrm{sinc}(\tfrac{\pi}{\ell}k)^\ell \cos(\pi k) \exp(2 \pi i k x)
\end{equation}
is the (normalized) B-spline of order $\ell$ with $\Vert N_\ell \Vert_{L_2} = 1$.
With the above normalization, we obtain $\Vert f \Vert_{L_2} = 1$.
Compared to the original function in \cite{KamPotVol2021}, we removed the highest-order term involving $N_6$ (which decays very fast) and reduced the dimension to $d=7$.

From the univariate series expansion \eqref{eq:FourierBSpline}, we can readily infer the Fourier coefficients of $f$. We observe an anisotropic decay, namely for $\mathbf{k}\in \mathbb{Z}^7$, we have that
\begin{equation}
    |\hat{f}_{\mathbf{k}}| \lesssim \prod\limits_{i=1}^7 (1+|k_i|)^{-r_i} \eqcolon a_{\mathbf{k}}, \qquad r_i = \begin{cases}
            2& i \in \{1,3,7\},\\
            4& i \in \{2,4,5,6\}.
    \end{cases}
\end{equation}
Using the hyperbolic cross $J_s$, we obtain truncation errors analogously to \eqref{eq:L2_trunc} and \eqref{eq:L_infty}.
The striking advantage over Example 1 shows up when estimating the best $n$-term approximation error $n^{-\sfrac{1}{2}} \sigma_n(f;\mathcal{B})_{\mathcal{A}_1}$.
In contrast to \eqref{eq:rearr}, the non-increasing rearrangement decays way faster.
In fact, using \cite[Thm.\ 3.4]{KSU21} with $\mathbf{q}= (1,...,1)$ shows that
\begin{equation}\label{eq:CoeffsFun2}
    a_n^{\ast} \lesssim n^{-2}\log^{4} n,\qquad n\in \mathbb{N}.
\end{equation}
This in turn implies that
\begin{equation} \label{eq:benchFunc2}
    n^{-\sfrac{1}{2}} \sigma_n(f;\mathcal{B})_{\mathcal{A}_1} \lesssim
    n^{-\sfrac{1}{2}} \sum\limits_{k>n}|a_k^{\ast}| \lesssim n^{-3/2}\log^{4} n.
\end{equation}

\begin{figure}[htb]
    \centering
    \captionof{table}{Example 2 with \rlasso \ (top) and \omp \ (bottom): The frequencies form a hyperbolic cross and the samples are drawn uniformly from $[0,1]^7$.
    All values are up to sampling randomness.
    }\label{tab:func4}

    \tabcolsep=0.15cm
    \begin{tabular}{c|ccccccc|c}
    \toprule
    $\lvert J \rvert$ | $m$ & 1K & 5K & 10K & 50K & 100K & 300K & 600K & trunc.\ error\\
    \midrule
    63K   & \num{2.51e-1} & \num{3.80e-2} & \num{2.96e-2} & \num{2.28e-2} & \num{2.10e-2} & \num{1.96e-2} & \num{1.88e-2} & \num{1.73e-2}\\
    198K  & \num{3.05e-1} & \num{3.57e-2} & \num{2.14e-2} & \num{1.56e-2} & \num{1.41e-2} & \num{1.26e-2} & \num{1.22e-2} & \num{1.07e-2}\\
    1.0M  & \num{5.31e-1} & \num{3.73e-2} & \num{1.13e-2} & \num{4.77e-3} & \num{4.10e-3} & \num{3.62e-3} & \num{3.43e-3} & \num{2.88e-3}\\
    2.1M  & \num{6.28e-1} & \num{3.35e-2} & \num{9.72e-3} & \num{2.98e-3} & \num{2.40e-3} & \num{2.08e-3} & \num{1.94e-03} & \num{1.56e-3}\\
    7.4M  & \num{6.65e-1} & \num{6.39e-2} & \num{1.25e-2} & \num{1.48e-3} & \num{1.08e-3} & \num{8.36e-4} & \num{7.54e-04} & \num{5.78e-4}\\
    \bottomrule
\end{tabular}

    \vspace{3mm}

    \tabcolsep=0.15cm
    \begin{tabular}{c|ccccccc|c}
    \toprule
    $\lvert J \rvert$ | $m$ & 1K & 5K & 10K & 50K & 100K & 300K & 600K & trunc.\ error\\
    \midrule
    63K   & \num{2.26e-1} & \num{4.31e-2} & \num{4.28e-2} & \num{4.40e-2} & \num{3.32e-2} & \num{2.51e-2} & \num{2.25e-2} & \num{1.73e-2}\\
    198K  & \num{2.68e-1} & \num{2.91e-2} & \num{2.66e-2} & \num{2.69e-2} & \num{2.30e-2} & \num{1.73e-2} & \num{1.52e-2} & \num{1.07e-2}\\
    1.0M  & \num{2.89e-1} & \num{1.13e-2} & \num{7.85e-3} & \num{6.78e-3} & \num{6.41e-3} & \num{5.07e-3} & \num{4.37e-3} & \num{2.88e-3}\\
    2.1M  & \num{3.37e-1} & \num{9.81e-3} & \num{4.71e-3} & \num{3.64e-3} & \num{3.53e-3} & \num{2.81e-3} & \num{2.43e-3} & \num{1.56e-3}\\
    7.4M  & \num{3.52e-1} & \num{9.79e-3} & \num{3.38e-3} & \num{1.38e-3} & \num{1.32e-3} & \num{1.08e-3} & \num{9.17e-4} & \num{5.78e-4}\\
    \bottomrule
\end{tabular}

\vspace{.2cm}

    \captionof{figure}{Example 2: Approximation error depending on the number of samples for the \rlasso \ and \omp \ algorithm with a large search space ($\lvert J \rvert = 27M$).
    The plot is in log scale.}
    \label{fig:Fkt2rLASSO}

    \vspace{2mm}

\begin{tikzpicture}
\begin{axis}[
    width=12cm,
    height=8cm,
    xlabel={Number of samples $m$ or terms $n$.},
    ylabel={$L_2$-Error},
    xmode=log,
    ymode=log,
    log basis x=10,
    log basis y=10,
    grid=both,
    grid style={dotted},
    mark=*,
]
\addplot coordinates {
    (500,      8.547153e-01)
    (1000,     7.031402e-01)
    (2000,     4.600763e-01)
    (4000,     1.700945e-01)
    (8000,     2.764872e-02)
    (16000,    5.859575e-03)
    (32000,    2.438509e-03)
    (64000,    1.039618e-03)
    (128000,   6.630426e-04)
    (256000,   4.799517e-04)
    (512000,   4.286203e-04)
};
\addlegendentry{\rlasso~approximant}

\addplot coordinates {
    (500,      9.244099e-01)
    (1000,     3.132426e-01)
    (2000,     7.191363e-02)
    (4000,     1.287952e-02)
    (8000,     6.509291e-03)
    (16000,    2.076376e-03)
    (32000,    1.044765e-03)
    (64000,    6.019930e-04)
    (128000,   5.525653e-04)
    (256000,   5.010097e-04)
    (512000,   4.673933e-04)
};
\addlegendentry{\omp~approximant}

\addplot[gray, mark=triangle*]coordinates {
     (500,      0.032380435616 + 0.00012723915212107113)
     (1000,     0.008383532986 + 0.00012723915212107113)
     (2000,     0.002924559172 + 0.00012723915212107113)
     (4000,     0.000990046072 + 0.00012723915212107113)
     (8000,     0.000304284971 + 0.00012723915212107113)
     (16000,    0.000057072208 + 0.00012723915212107113)
};
 \addlegendentry{best $n$-term}

\addplot[
    domain=500:32000,
    thick, no marks,
]
{3e-2 * x^(-1.5) * ln(x)^4};
\addlegendentry{$x^{-{\sfrac{3}{2}}}\log(x)^4$}
\end{axis}
\end{tikzpicture}
\end{figure}

Hence, our theory yields an even better behavior compared to Example 1 when using the~\rlasso\\ or~\omp~decoder with $m\approx n(\log n)^3$ many samples.
For \rlasso, we use the parameter $\lambda = 2m^{\sfrac{1}{2}}$.
Numerical approximation results for several combinations of $m$ and $J_s$ are given in \Cref{tab:func4}.
As before, we indicate the truncation error on the right and use the ground-truth Fourier coefficients to compute the error.
The behavior with respect to $m$ and $\vert J_s \vert$ is similar to our first example.

In \Cref{fig:Fkt2rLASSO}, we plot the errors for the \rlasso \ and \omp \ decoders with $\vert J_s \vert = 27M$.
We also include the best $n$-term approximation of the sequence \eqref{eq:CoeffsFun2} in $\ell_2$ together with the asymptotic bound in \eqref{eq:benchFunc2} as benchmark.
Again, the best $n$-term error matches the predicted asymptotic rates although the pre-asymptotic effects are much stronger here.
For this example, \rlasso \ struggles for small numbers of samples $m$, which is also visible in \Cref{tab:func4}.
For larger $m$, the two decoders behave relatively similar.

\begin{remark}[Least squares]
    To get a similar decay with the least squares decoder, one has to incorporate knowledge about the optimal linear subspaces for  approximation in $L_2$, which is determined by an anisotropic hyperbolic cross index set, see \cite[Sect.\ 2.3]{KSU21}, given by
    \begin{equation}
        J_{\mathbf{r},s} \coloneqq \Big\{\mathbf{k}\in \Z^7 \colon \; \prod\limits_{i=1}^7 (1+|k_i|)^{r_i}\leq s\Big\}.
    \end{equation}
    Note that the usual isotropic hyperbolic cross index set $J_s$  would lead to a non-optimal bound.
    Neither \rlasso~nor \omp~uses such a priori information.
\end{remark}

\subsubsection{The tensorized Chebyshev system}
Now, we consider functions in $L_2([-1,1]^d, {\omega}_d)$, where $\omega_d$ is the $d$-dimensional normalized Chebyshev measure
\begin{equation}\label{prob_measure_Cheb}
\omega_d(\mathbf x) = \prod_{j=1}^d \pi^{-1}(1-x_j^2)^{-\sfrac12}.
\end{equation}
An orthonormal basis of $L_2([-1,1], {\omega})$ is given by the Chebyshev polynomials
\begin{equation}
    T_k(x) = \sqrt{2}^{\min(1, k)}
 \cos (k \arccos x), \quad k \in \N_0.
\end{equation}
These can be equivalently characterized as eigenfunctions of the Sturm-Liouville problem
\begin{equation}
    \left( \sqrt{1-x^2} T'_k(x) \right)' + \frac{k^2}{\sqrt{1-x^2}} T_k(x) = 0.
\end{equation}
Further, an orthonormal basis of $L_2([-1,1]^d, {\omega}_d)$ is given by the tensor product functions
\begin{equation}
    T_{\mathbf{k}} (\mathbf x) = \prod_{j=1}^{d} T_{k_j} (x_j)= \prod_{j=1}^{d} \left( \sqrt{2}^{\min(1, k_j)}
 \cos (k_j \arccos x_j) \right), \quad k \in \mathbb{N_0}^d.
\end{equation}
A detailed survey on the Chebyshev polynomials can be found in the monograph \cite{MasonHandscomb2003}.
Note that due to the density $\omega_d$, we also need to sample $\mathbf X \sim \omega_d$.
As search space, we again use the (restricted) hyperbolic cross $J_{s} \coloneqq \{\mathbf{k}\in\mathbb{N}^d \colon \; \Pi(\mathbf{k})\leq s\}$.
For \rlasso, we use $\lambda = m^{\sfrac{1}{2}}$.
With the same reasoning as for the Fourier basis, we set $K=\min(m, \vert J \vert, 20000)$ for \omp.
For \cosamp \ we set $n = \min(m/4, \vert J \vert, 20000)$ in order to ensure that the least squares problem remains sufficiently well-posed.

\begin{figure}[p]
    \centering
    \captionof{table}{Example 3 with \rlasso \ (top), \omp \ (middle) and \cosamp \ (bottom): The search space is a hyperbolic cross and the samples are drawn according to the Chebyshev density on $[-1,1]^6$.
    Errors are computed using the ground-truth coefficients.
    All values are up to sampling randomness.}\label{tab:func2}

    \tabcolsep=0.15cm
    \begin{tabular}{c|cccccc|c}
    \toprule
    $\lvert J \rvert$ | $m$ & 1K & 5K & 10K & 50K & 100K & 200K  & trunc.\ error\\
    \midrule
    1K  & \num{5.35e-2} & \num{2.12e-2} & \num{1.87e-2} & \num{1.58e-2} & \num{1.55e-2} & \num{1.47e-2}  & \num{1.34e-2}\\
    3.8K  & \num{6.19e-2} & \num{6.92e-3} & \num{4.31e-3} & \num{3.19e-3} & \num{3.00e-3} & \num{2.81e-3}  &  \num{2.41e-3} \\
    12K  & \num{1.42e-1} & \num{5.15e-3} & \num{2.06e-3} & \num{1.14e-3} & \num{1.01e-3} & \num{9.51e-4} & \num{7.47e-4} \\
    23K & \num{1.14e-1} & \num{9.11e-3} & \num{1.73e-3} & \num{4.01e-4} & \num{3.44e-4} & \num{3.08e-4} & \num{2.30e-4}\\
    52K & \num{1.93e-1} & \num{1.15e-2} & \num{2.97e-3} & \num{1.85e-4} & \num{1.39e-4} & \num{1.20e-4} & \num{7.98e-5}\\
    150K & \num{1.90e-1} & \num{1.70e-2} & \num{5.58e-3} & \num{1.18e-4} & \num{4.28e-5} & \num{2.88e-5} & \num{1.53e-5}\\
    \bottomrule
\end{tabular}

    \vspace{3mm}
    \begin{tabular}{c|cccccc|c}
    \toprule
    $\lvert J \rvert$ | $m$ & 1K & 5K & 10K & 50K & 100K & 200K  & trunc.\ error\\
    \midrule
    1K  & \num{1.40e-1} & \num{2.07e-2} & \num{1.96e-2} & \num{1.58e-2} & \num{1.50e-2} & \num{1.45e-2}  & \num{1.34e-2}\\
    3.8K  & \num{1.24e-1} & \num{9.91e-3} & \num{4.76e-3} & \num{3.14e-3} & \num{2.94e-3} & \num{2.81e-3}  &  \num{2.41e-3} \\
    12K  & \num{1.69e-1} & \num{1.32e-2} & \num{4.50e-3} & \num{1.24e-3} & \num{1.05e-3} & \num{9.43e-4} & \num{7.47e-4} \\
    23K & \num{2.10e-1} & \num{1.75e-2} & \num{3.26e-3} & \num{5.11e-4} & \num{3.72e-4} & \num{3.18e-4} & \num{2.30e-4}\\
    52K & \num{1.83e-1} & \num{1.84e-2} & \num{3.97e-3} & \num{2.42e-4} & \num{1.81e-4} & \num{1.58e-4} & \num{7.98e-5}\\
    150K & \num{1.84e-1} & \num{2.15e-2} & \num{4.96e-3} & \num{2.22e-4} & \num{1.44e-4} & \num{1.27e-4} & \num{1.53e-5}\\
    \bottomrule
\end{tabular}

    \vspace{3mm}
    \begin{tabular}{c|cccccc|c}
    \toprule
    $\lvert J \rvert$ | $m$ & 1K & 5K & 10K & 50K & 100K & 200K  & trunc.\ error\\
    \midrule
    1K  & \num{1.03e-1} & \num{2.54e-2} & \num{1.92e-2} & \num{1.58e-2} & \num{1.50e-2} & \num{1.46e-2}  & \num{1.34e-2}\\
    3.8K  & \num{1.82e-1} & \num{9.68e-3} & \num{4.76e-3} & \num{3.16e-3} & \num{2.92e-3} & \num{2.76e-3}  &  \num{2.41e-3} \\
    12K  & \num{2.21e-1} & \num{2.32e-2} & \num{3.28e-3} & \num{1.24e-3} & \num{1.05e-3} & \num{9.46e-4} & \num{7.47e-4} \\
    23K & \num{4.61e-1} & \num{3.24e-2} & \num{5.42e-3} & \num{4.87e-4} & \num{3.71e-4} & \num{3.17e-4} & \num{2.30e-4}\\
    52K & \num{5.92e-1} & \num{7.20e-2} & \num{8.54e-3} & \num{4.03e-4} & \num{1.87e-4} & \num{1.25e-4} & \num{7.98e-5}\\
    150K & \num{6.84e-1} & \num{1.11e-1} & \num{2.99e-2} & \num{6.44e-4} & \num{1.58e-4} & \num{7.25e-5} & \num{1.53e-5}\\
    \bottomrule
    \end{tabular}
    \vspace{.2cm}

    \captionof{figure}{Example 3: Approximation error depending on the number of samples for the \rlasso, \omp \ and \cosamp \ algorithm with fixed $J$.
    The plot is in log scale.}

    \vspace{2mm}

    \label{fig:Fkt2CoSaMP}
    \centering
    \begin{tikzpicture}
    \begin{axis}[
    width=12cm,
    height=8cm,
    xlabel={Number of samples $m$ or terms $n$.},
    ylabel={$L_2$-Error},
    xmode=log,
    ymode=log,
    log basis x=10,
    log basis y=10,
    grid=both,
    grid style={dotted},
    mark=*,
]

\addplot coordinates {
    (4000,     3.188542e-02)
    (8000,     1.209094e-02)
    (16000,    1.665870e-03)
    (32000,    3.298991e-04)
    (64000,    8.156026e-05)
    (128000,   3.556109e-05)
    (256000,   2.631823e-05)
};
\addlegendentry{\rlasso~approximant}

\addplot coordinates {
    (4000,     2.496526e-02)
    (8000,     9.501535e-03)
    (16000,    2.354494e-03)
    (32000,    4.531720e-04)
    (64000,    1.713265e-04)
    (128000,   1.356910e-04)
    (256000,   1.241120e-04)
};
\addlegendentry{\omp~approximant}

\addplot[mark=diamond*]coordinates {
    (4000,     1.806655e-01)
    (8000,     8.351270e-02)
    (16000,    9.821614e-03)
    (32000,    1.841629e-03)
    (64000,    4.598911e-04)
    (128000,   2.037179e-04)
    (256000,   4.974523e-05)
};
\addlegendentry{\cosamp~approximant}

\addplot[gray, mark=triangle*]coordinates {
    (4000,     0.000822020229)
    (8000,     0.000244716997)
    (16000,    0.000079548678)
    (32000,    0.000023501134)
    (64000,    0.000007010537)
    (128000,   0.000002021078)
    (256000,   0.000000581411)
};
\addlegendentry{best $n$-term}
\end{axis}
\end{tikzpicture}
\end{figure}

\paragraph{Example 3, a non-periodic function.}

Following the authors of \cite{BarLutNag2023}, we approximate the function $f \colon [-1,1]^d \to \R$ with $f(\mathbf{x}) = \prod_{j = 1}^d g(x_j)$, where
\begin{equation}
    g(x) = \sqrt{\frac{1024\pi}{367\pi-256}}\begin{cases}
        -\frac{x^2}{4} - \frac{x}{2} + \frac{1}{2} & -1 \leq x \leq 0 \\
        \frac{x^2}{8} - \frac{x}{2} + \frac{1}{2} & 0 \leq x \leq 1\\
    \end{cases}.
\end{equation}
Using the substitution $x = \sin(\theta)$, we obtain (with a lengthy but elementary derivation) that \smash{$\Vert f \Vert^2_{L_2([-1,1]^d,\omega_d)} = 1$}.
With another lengthy calculation \cite{BarLutNag2023}, we obtain its Chebyshev expansion
\begin{equation}
    g = \sqrt{\frac{1024\pi}{367\pi-256}}\biggl(\frac{15}{32}T_0 - \frac{\pi-1}{2\pi\sqrt{2}}T_1 -\frac{1}{32\sqrt{2}}T_2-\frac{3}{2\pi\sqrt{2}}\sum\limits_{k=3}^{\infty}\frac{\sin(k\pi/2)}{k(k^2-4)}T_k\biggr).
\end{equation}
Similarly to the Fourier system, we obtain estimates for the best $n$-term approximation of $f$ with respect to the tensorized Chebyshev system via rearrangement. We first observe that
\begin{equation}
    |[f]_{\mathbf{k}}| \lesssim \prod\limits_{i=1}^d (1+k_i)^{-3}\eqqcolon a_{\mathbf{k}},\quad \mathbf{k} \in \mathbb{N}_0^d.
\end{equation}
Using again \cite{KSU15}, the non-increasing rearrangement of  $([f]_{\mathbf{k}})_{\mathbf{k}\in \mathbb{N}_0^d}$ satisfies
\begin{equation}
    a_n^{\ast} \asymp n^{-3}(\log n)^{3(d-1)},\quad n\in \mathbb{N}.
\end{equation}
This implies
\begin{equation} \label{eq:bestmterm3}
    \sigma_n(f;\mathcal{B})_{L_2(\omega_d)} \lesssim \biggl(\sum\limits_{k>n} \lvert a_k^{\ast}\rvert^2 \biggr)^{1/2} \asymp n^{-5/2}(\log n)^{3(d-1)}
\end{equation}
and similarly
\begin{equation}
    n^{-\sfrac{1}{2}} \sigma_n(f;\mathcal{B})_{\mathcal{A}_1} \lesssim n^{-\sfrac{1}{2}} \sum\limits_{k>n} \lvert a_k^{\ast}\rvert \asymp n^{-5/2}(\log n)^{3(d-1)}.
\end{equation}
Choosing the hyperbolic cross search space large enough we ensure that the corresponding projection error is of the order \eqref{eq:bestmterm3}.
This can be realized by choosing $s \asymp n^\kappa$ with $\kappa>1$.
Hence, from \Cref{thm:Error_Recovery_Ops} we obtain for this function
\begin{equation}
	\lVert f - H_{\mathbf{X}}(f)\rVert_{L_2(\omega_d)} \leq n^{-\sfrac{1}{2}}\sigma_n(f;\mathcal{B})_{\mathcal{A}_1} + \lVert f-P_{J}f\rVert_{L_\infty} \lesssim n^{-5/2}(\log n)^{3(d-1)},
\end{equation}
where $H_{\mathbf{X}}$  is either the \rlasso~or~\omp~decoder with $m\approx n(\log n)^3$ many samples.

The approximation results for $d=6$ are given in \Cref{tab:func2}.
Again, we indicate the (approximated) truncation error on the right.
As for the previous examples, we observe that the errors decay with more samples before they eventually saturate at the same order of magnitude as the truncation error.
Numerically, we compare the convergence rates for $d=6$ and $J$ with $\vert J \vert = 150K$ in \Cref{fig:Fkt2CoSaMP}.
Here, the truncation error behaves much better than the theoretical guaranteed rate $n^{-\sfrac{5}{2}}(\log n)^{15}$.
Interestingly, also all three non-linear decoders roughly match the decay rate of the best $n$-term approximation, which might be a pre-asymptotic effect.
For \omp, we again observe a saturation effect for $m \geq 100K$ samples.

\begin{figure}[htbp]
    \centering
    \captionof{table}{Example 4 with \rlasso \ (top), \omp \ (middle) and CoSaMP (bottom): The frequencies are chosen from the hyperbolic cross and the samples are drawn according to the Chebyshev density on $[-1,1]^6$.}
    \label{tab:func3}

    \tabcolsep=0.15cm
    \begin{tabular}{c|ccccccc}
        \toprule
        $\lvert J \rvert$ | $m$ & 1K & 5K & 10K & 50K & 100K & 300K & 600K\\
        \midrule
       1K  & \num{1.53e-1} & \num{5.59e-2} & \num{4.46e-2} & \num{3.86e-2} & \num{3.80e-2} & \num{3.77e-2} & \num{3.76e-2} \\
    2.8K   & \num{1.93e-1} & \num{4.79e-2} & \num{3.62e-2}  & \num{2.74e-2} & \num{2.67e-2} & \num{2.61e-2} & \num{2.59e-2} \\
    12K    & \num{2.89e-1} & \num{6.75e-2} & \num{3.14e-2}  & \num{8.67e-3} & \num{7.37e-3} & \num{6.62e-3} & \num{6.45e-3} \\
    23K   & \num{2.82e-1}  & \num{9.45e-2} & \num{3.49e-2} & \num{6.87e-3} & \num{5.04e-3} & \num{4.13e-3} & \num{3.94e-3} \\
    70K   & \num{3.65e-1}  & \num{1.20e-1} & \num{5.40e-2} & \num{8.03e-3} & \num{3.67e-3} & \num{1.92e-3} & \num{1.63e-3} \\
    231K  & \num{3.76e-1}  & \num{1.46e-1} & \num{7.69e-2} & \num{1.47e-2} & \num{5.24e-3} & \num{1.27e-3} & \num{7.69e-4} \\
        \bottomrule
    \end{tabular}

    \vspace{3mm}

    \tabcolsep = .15cm
    \begin{tabular}{c|ccccccc}
        \toprule
        $\lvert J \rvert$ | $m$ & 1K & 5K & 10K & 50K & 100K & 300K & 600K\\
        \midrule
        1K & \num{3.76e-1} & \num{4.49e-2} & \num{4.06e-2}  & \num{3.82e-2} & \num{3.78e-2} & \num{3.76e-2} & \num{3.75e-2} \\
        2.8K   & \num{2.52e-1} & \num{4.65e-2} & \num{3.42e-2} & \num{2.70e-2} & \num{2.64e-2} & \num{2.60e-2} & \num{2.59e-2} \\
        12K  & \num{3.11e-1} & \num{6.05e-2} & \num{3.70e-2} & \num{7.89e-3} & \num{6.89e-3} & \num{6.50e-3} & \num{6.39e-3} \\
        23K & \num{2.93e-1} & \num{7.44e-2} & \num{3.96e-2}  & \num{6.05e-3} & \num{4.65e-3} & \num{4.00e-3} & \num{3.88e-3} \\
        70K & \num{2.92e-1} & \num{1.11e-1} & \num{4.64e-2} & \num{6.79e-3} & \num{4.00e-3} & \num{3.01e-3} & \num{2.83e-3} \\
        231K &\num{3.16e-1} & \num{1.56e-1} & \num{5.30e-2} & \num{1.01e-2} & \num{4.79e-3} & \num{3.01e-3} & \num{2.84e-3} \\
        \bottomrule
    \end{tabular}
        \vspace{3mm}

        \tabcolsep=0.15cm
\begin{tabular}{c|ccccccc}
        \toprule
        $\lvert J \rvert$ | $m$ & 1K & 5K & 10K & 50K & 100K & 300K & 600K\\
        \midrule
       1K  &  \num{1.93e-1} & \num{4.54e-2} & \num{4.13e-2} & \num{3.80e-2} & \num{3.78e-2} & \num{3.76e-2} & \num{3.75e-2} \\
    2.8K  & \num{3.03e-1} & \num{4.59e-2} & \num{3.45e-2} & \num{2.69e-2} & \num{2.64e-2} & \num{2.60e-2} & \num{2.59e-2} \\
    12K  & \num{3.75e-1} & \num{6.86e-2} & \num{3.30e-2} & \num{7.87e-3} & \num{6.98e-3} & \num{6.47e-3} & \num{6.39e-3} \\
    23K & \num{4.69e-1} & \num{1.03e-1} & \num{5.79e-2} & \num{6.55e-3} & \num{4.62e-3} & \num{4.00e-3} & \num{3.87e-3} \\
    70K & \num{5.88e-1} & \num{2.00e-1} & \num{1.19e-1}  & \num{8.50e-3} & \num{3.37e-3} & \num{1.78e-3} & \num{1.57e-3} \\
    231K  & \num{7.35e-1} & \num{2.92e-1} & \num{2.01e-1} & \num{1.61e-2} & \num{8.92e-3} & \num{1.38e-3} & \num{7.59e-4} \\
        \bottomrule
    \end{tabular}
    \vspace{.2cm}

    \captionof{figure}{Example 4: Approximation error depending on the number of samples for the \rlasso, \omp \ and \cosamp \ algorithm with fixed $J$.
    The plot is in log scale.}
    \label{fig:Fkt4CoSaMP}

    \vspace{2mm}

    \centering
    \begin{tikzpicture}
    \begin{axis}[
    width=12cm,
    height=8cm,
    xlabel={Number of samples $m$.},
    ylabel={$L_2$-Error},
    xmode=log,
    ymode=log,
    log basis x=10,
    log basis y=10,
    grid=both,
    grid style={dotted},
    mark=*,
]

\addplot coordinates {
    (1000,     0.4828723669052124)
    (2000,     0.2446237951517105)
    (4000,     0.1791776567697525)
    (8000,     0.11784767359495163)
    (16000,    0.056535691022872925)
    (32000,    0.028220273554325104)
    (64000,    0.012469477951526642)
    (128000,   0.005053224042057991)
    (256000,   0.0019211926264688373)
    (512000,   0.0008358341292478144)
};
\addlegendentry{\rlasso~approximant}

\addplot coordinates {
    (1000,     0.3306838870048523)
    (2000,     0.27405351400375366)
    (4000,     0.20612893998622894)
    (8000,     0.06756797432899475)
    (16000,    0.043454889208078384)
    (32000,    0.02173588052392006)
    (64000,    0.007969559170305729)
    (128000,   0.005617085844278336)
    (256000,   0.005132284481078386)
    (512000,   0.004953311290591955)
}; 
\addlegendentry{\omp~approximant}

\addplot[mark=diamond*]coordinates {
    (1000,     0.7996897101402283)
    (2000,     0.4913075268268585)
    (4000,     0.4152829051017761)
    (8000,     0.2713572084903717)
    (16000,    0.11322832852602005)
    (32000,    0.07323053479194641)
    (64000,    0.021536849439144135)
    (128000,   0.009170484729111195)
    (256000,   0.0025865575298666954)
    (512000,   0.0007478544139303267)
};
\addlegendentry{\cosamp~approximant}

\addplot[
    domain=1000:512000,
    thick, no marks,
]
{1.5e-6 * x^(-1.5) * ln(x)^10};
\addlegendentry{$x^{-{\sfrac{3}{2}}}\log(x)^{10}$}
\end{axis}
\end{tikzpicture}
\vspace{40pt}
\end{figure}

\paragraph{Example 4, a kink function.}

Following \cite{EggMinUll25}, we consider the function $f \colon [-1,1]^d \to \R$ with
\begin{equation}
    f(\mathbf{x}) = \prod_{i = 1}^d \frac{\vert 8 x_i -6 - w_i\vert + c_i}{1 + c_i},
\end{equation}
where we used \(c_i = 1\) and \(w_i = 0.4\). We can explicitly calculate the $L_2$-norm of this function to be
\begin{equation}\label{eq:Fun4Norm}
    \lVert f\rVert_{L_2}^2 = \biggl(\frac{32}{5\pi}\cdot\arcsin\biggl(\frac{4}{5}\biggr) + \frac{24}{5\pi} + \frac{1849}{100}\biggr)^d.
\end{equation}
Since no ground-truth Chebyshev coefficients are known in the literature for this function, we have to estimate the approximation error based on Monte-Carlo estimation of the $L_2(\mu)$-norm.
This is done with $10^6$ samples, and we report the average of 5 estimations.
The observed standard deviation is always less than 1\%.
The relative approximation results (we normalize $f$ according to \eqref{eq:Fun4Norm}) are given in \Cref{tab:func3} for $d=6$.

In principle, we see the same trend as before regarding the number of samples $m$ and the size of $J_s$.
We also visualize the convergence behavior for $d=6$ and $J_s$ with $\vert J_s \vert = 231K$ in \Cref{fig:Fkt4CoSaMP}.
Again, all three decoders have roughly the same asymptotic decay.
This decay is determined by the fact that $f$ is a tensor product of piecewise linear functions, which implies the mixed Besov regularity \smash{$\mathbf{B}^2_{1,\infty}([-1,1]^d)$}.
As in Example 1, this leads to the rate $n^{-\sfrac{3}{2}}(\log n)^{2(d-1)}$, but this time with $d=6$ instead of $d=5$.
For \omp, we observe the same saturation effect as before.

\subsection{General discussion}
For all examples, the errors decay with increasing number of samples $m$ and expanding search space $J$.
However, there are certain exceptions to that rule.
First, for $m < \lvert J \rvert$ fixed, we observe that the error often increases with $\lvert J \rvert$. Neither is this surprising since the problem becomes more under-determined, nor is it a contradiction to our theoretical observations since $\lvert J \rvert$ appears in the required oversampling bound \eqref{eq:NumberSamples}. Therefore, the number of samples required to achieve the same error increases.
Second, the errors stagnate as soon as they are in the order of the truncation error, which is, of course, exactly what we expect due to our bounds in \Cref{thm:Error_Recovery_Ops}, where the term \(\lVert f-P_Jf\rVert_{L_\infty}\) is an upper bound to the truncation error. This effect seems to be more prevalent for \omp, while \cosamp \ and the best \(n\)-term approximation perform much better with the same number of non-zero entries in the approximant.

One might also wonder why the errors do not converge to the truncation error.
Here, it should be noted that the samples are not drawn from the truncated function itself.
Consequently, a certain level of mismatch is to be expected.
Third, we observe that the truncation errors do not decay anymore as soon as we reach the numerical precision limit.
In particular, due to the GPU computing approach, we generally only have the single precision limit.

%% file: sections/appendix.tex
\appendix

\section{RIP and NSP with \rlasso \ recovery guarantees}\label{app:RIPNSP}

For convenience, we recite the key theorems from the literature adapted to our notation and requirements.
This also includes some simplifications leading to weaker statements.

\begin{proof}[Proof of \Cref{thm:RIP_implies_NSP}]
Let $\mathbf{v} \in \C^N \setminus \{\boldsymbol{0}\}$.
We consider the index set $J_n(\mathbf{v})$ of the $n$ largest entries of $\mathbf{v}$ in absolute value, and partition $[N]$ into $S_k\coloneqq J_{(k+1)n}(\mathbf{v})\setminus J_{kn}(\mathbf{v})$ for $0\leq k<\lceil\sfrac{N}{n}\rceil$.
It suffices to show the $\ell_2$-robust NSP for $S_0$ since \eqref{eq:NSP} holds for all $S\subset[N]$ with $\lvert S\rvert \leq n$ if it holds for $S_0$, namely the $n$ largest entries.
Since $S_0\cap S_k=\emptyset$ for all $k>0$, the $n$-sparse vectors $\mathbf{v}_{S_0}$ and $\mathbf{v}_{S_k}$ have disjoint support, 
and \cite[Prop.\ 6.3]{FoRa13} gives us
\begin{gather}\label{eq:Est_scalar}
    \lvert\langle \mathbf{A}\mathbf{v}_{S_0}, \mathbf{A}\mathbf{v}_{S_k} \rangle\rvert\leq \delta_{2n} \lVert\mathbf{v}_{S_0}\rVert_{\ell_2}  \lVert\mathbf{v}_{S_k}\rVert_{\ell_2}.
\end{gather}
Moreover, since $\lvert v_i\rvert\leq\lvert v_j\rvert$ for all $i\in S_k$ and $j\in S_{k-1}$, we have $\lvert v_i\rvert \leq n^{-1}\lVert\mathbf{v}_{S_{k-1}}\rVert_{\ell_1}$ for all $k\geq 1$ and $i\in S_k$.
Summing over all $i\in S_k$, we see that
\begin{gather}\label{eq:Est_norm}
    \lVert\mathbf{v}_{S_k}\rVert_{\ell_2} = \Bigl(\sum_{i\in S_k}\lvert v_i\rvert^2 \Bigr)^{\sfrac{1}{2}} \leq n^{\sfrac{1}{2}}\big(n^{-1}\lVert\mathbf{v}_{S_{k-1}}\rVert_{\ell_1}\big)= n^{-\sfrac{1}{2}}\lVert\mathbf{v}_{S_{k-1}}\rVert_{\ell_1}.
\end{gather}
With \eqref{eq:Est_scalar} and \eqref{eq:Est_norm} at hand, we estimate as follows
\begin{align}
\begin{split}
\lVert\mathbf{v}_{S_0}\rVert_{\ell_2}^2 & \leq \frac{1}{1-\delta_{2n}}\lVert\mathbf{A}\mathbf{v}_{S_0}\rVert_{\ell_2}^2=\frac{1}{1-\delta_{2n}} \Bigl\langle \mathbf{A}\mathbf{v}_{S_0}, \mathbf{A}\mathbf{v} - \sum \limits_{k\geq 1} \mathbf{A}\mathbf{v}_{S_k}\Bigr\rangle\\
&\leq\frac{1}{1-\delta_{2n}} \Big( \lvert\langle \mathbf{A}\mathbf{v}_{S_0}, \mathbf{A}\mathbf{v} \rangle\rvert + \sum \limits_{k\geq 1} \lvert\langle \mathbf{A}\mathbf{v}_{S_0},\mathbf{A}\mathbf{v}_{S_k}\rangle\rvert \Big)\\
&\leq\frac{1}{1-\delta_{2n}} \Big( \lVert\mathbf{A}\mathbf{v}_{S_0}\rVert_{\ell_2} \lVert\mathbf{A}\mathbf{v}\rVert_{\ell_2} + \delta_{2n} \lVert\mathbf{v}_{S_0}\rVert_{\ell_2} \sum_{k\geq 1}  \lVert\mathbf{v}_{S_k}\rVert_{\ell_2} \Big)\\
&\leq\frac{1}{1-\delta_{2n}} \Big( \sqrt{1+\delta_{2n}}\lVert\mathbf{v}_{S_0}\rVert_{\ell_2} \lVert\mathbf{A}\mathbf{v}\rVert_{\ell_2} +\delta_{2n} \lVert\mathbf{v}_{S_0}\rVert_{\ell_2}n^{-\sfrac{1}{2}}\sum \limits_{k\geq 1} \lVert\mathbf{v}_{S_{k-1}}\rVert_{\ell_1}\Big).
\end{split}
\end{align}
Division by $\lVert\mathbf{v}_{S_0}\rVert_{\ell_2}$ yields
\begin{gather}
    \lVert\mathbf{v}_{S_0}\rVert_{\ell_2}\leq\frac{\sqrt{1+\delta_{2n}}}{1-\delta_{2n}}\lVert\mathbf{A}\mathbf{v}\rVert_{\ell_2}+\frac{\delta_{2n}}{1-\delta_{2n}} n^{-\sfrac{1}{2}} \big(\lVert\mathbf{v}_{S_0}\rVert_{\ell_1}+ \lVert\mathbf{v}_{S_0^\mathrm{C}}\rVert_{\ell_1} \big).
\end{gather}
Due to the fact that $\lVert\mathbf{v}_{S_0}\rVert_{\ell_1} \leq n^{\sfrac{1}{2}} \lVert\mathbf{v}_{S_0}\rVert_{\ell_2}$, we can rearrange this to
\begin{gather}\label{eq:NSP_proof}
\lVert\mathbf{v}_{S_0}\rVert_{\ell_2}\leq\biggl(1-\frac{\delta_{2n}}{1-\delta_{2n}}\biggr)^{-1} \biggl (\frac{\delta_{2n}}{1-\delta_{2n}} n^{-\sfrac{1}{2}}\lVert\mathbf{v}_{S_0^\mathrm{C}}\rVert_{\ell_1} +  \frac{\sqrt{1+\delta_{2n}}}{1-\delta_{2n}}\lVert\mathbf{A}\mathbf{v}\rVert_{\ell_2} \biggr).
\end{gather}
Comparing \eqref{eq:NSP} and \eqref{eq:NSP_proof}, the assertion follows by setting $\varrho$ and $\tau$ accordingly. 
\end{proof}

The following theorem gives bounds for the \rlasso \ minimization problem \eqref{eq:MinimizationProblem}. Such bounds were first shown by Petersen and Jung in \cite[Thm.\ 3.1]{PeJu22} for real matrices. The version below is adapted to complex matrices.

\begin{theorem}[{\cite[Thm.\ 6.29]{AdBrWe22}}]\label{thm:LassoNSP}
Let $\mathbf{A}\in\mathbb{C}^{m\times N}$ with $\mathbf{A}\in\operatorname{NSP}(n,2,\lVert\cdot\rVert_{\ell_2},\varrho,\tau)$. If
\begin{equation}\label{eq:lambdaLowerBound}
    \lambda\geq \frac{3+\varrho}{1+\varrho}\tau n^{\sfrac{1}{2}},
\end{equation}
then every minimizer
\begin{equation}
    \mathbf{r(\mathbf{y})} \in \argmin_{\mathbf{z}\in\mathbb{C}^N}\lVert\mathbf{z}\rVert_{\ell_1}+\lambda\lVert\mathbf{y}-\mathbf{A}\mathbf{z}\rVert_{\ell_2}
\end{equation}
obeys, for all $\mathbf{v}\in\mathbb{C}^N$,
\begin{equation}\label{eq:LassoNSP}
    \lVert\mathbf{v}-\mathbf{r}(\mathbf{y})\rVert_{\ell_p}\leq C\bigl(n^{\sfrac{1}{p}-1}\sigma_n(\mathbf{v})_{\ell_1} + \bigl(\lambda n^{\sfrac{1}{p}-1}+n^{1-\sfrac{p}{2}}\bigr)\lVert\mathbf{y}-\mathbf{A}\mathbf{v}\rVert_{\ell_2}\bigr).
\end{equation}
for $p\in\{1,2\}$ and a constant $C>0$ depending on $\varrho$ and $\tau$.
\end{theorem}

There is a similar result to \Cref{thm:RIP} below in \cite[Thm.\ 2.3]{BDJR21}. 
However, we remain skeptical whether it shows the correct dependency on the parameter $\varepsilon>0$ (which, however, is only of minor importance to us). 

\begin{theorem}[{\cite[Thm.\ 6.15]{AdBrWe22}}]\label{thm:RIP}
    Let $\mathcal{B}=\{b_j\colon \; j\in I\}\subset L_2(\mu)$ be a BOS and $B=\sup_{j\in I}\lVert b_j\rVert_{L_\infty}$.
    Then there is a constant $\alpha>0$ such that for all $n\in\mathbb{N}$, $J\subset I$ finite, $\varepsilon>0$, $\gamma\in(0,1)$ and
    \begin{equation}
        m\geq \alpha B^2\varepsilon^{-2}n w_n\bigl(\varepsilon^{-4}w_n\log(2\lvert J\rvert)+ \varepsilon^{-1}\log\bigl(2(\gamma\varepsilon)^{-1}w_n\bigr)\bigr), \qquad \text{where } w_n = \log(2B^2\varepsilon^{-2}n),
    \end{equation}
    the matrix $\mathbf{A}$ defined in $\eqref{eq:matrix}$ satisfies the RIP with constant $\delta_{n}<\varepsilon$ with probability at least $1-\gamma$.
\end{theorem}

\section[Proofs for Section 3]{Proofs for \Cref{sec:InstanceOptVec}}\label{sec:ProofsInstance}
For our proofs, we require two results from the literature, which are given below.
\begin{lemma}[{\cite[Lem.\ 7.4.1]{DTU18}}]\label{thm:Stechkin}
For any $q\geq p\geq1$ and any $\mathbf{z}\in\C^N$ we have
\begin{equation}
    \sigma_n(\mathbf{z})_{\ell_q}\leq n^{\sfrac{1}{q}-\sfrac{1}{p}}\lVert \mathbf{z} \rVert_{\ell_p}.
\end{equation}
\end{lemma}

\begin{corollary}\label{thm:Sigmas}
Let $\mathbf{z}\in\C^N$ and $q\geq p\geq 1$. Then $\sigma_{3n}(\mathbf{z})_{\ell_q}\leq n^{\sfrac{1}{q}-\sfrac{1}{p}}\sigma_{2n}(\mathbf{z})_{\ell_p}$.
\end{corollary}

\begin{proof}
Let $S$ be the index set of the $2n$ largest absolute entries of $\mathbf{z}$. By \Cref{thm:Stechkin}, we get
\begin{equation}
    \sigma_{3n}(\mathbf{z})_{\ell_q}=\sigma_n(\mathbf{z}_{S^{\mathrm{C}}})_{\ell_q}\leq n^{\sfrac{1}{q}-\sfrac{1}{p}}\lVert \mathbf{z}_{S^{\mathrm{C}}}\rVert_{\ell_p}=n^{\sfrac{1}{q}-\sfrac{1}{p}}\sigma_{2n}(\mathbf{z})_{\ell_p}.\qedhere
\end{equation}
\end{proof}

\begin{lemma}[{\cite[Lem.\ 2.3]{FPRU10}}]\label{thm:SubsetCombinations}
Let $n,N\in\N$ with $n<N$. Then there is a family of subsets $\mathcal{U}\subset\mathcal{P}(\{1,\ldots,N\})$ satisfying 
\begin{itemize}
    \item $\lvert S\rvert=n$ for all $S\in\mathcal{U}$,
    \item $\lvert S\cap T\rvert<\sfrac{n}{2}$ for all $S,T\in\mathcal{U}$ with $S\neq T$,
    \item $\lvert\mathcal{U}\rvert\geq(\sfrac{N}{4n})^{\sfrac{n}{2}}$.
\end{itemize}
\end{lemma}
Below, we provide the proofs for \Cref{sec:InstanceOptimality}.
\begin{proof}[Proof of \Cref{prop:IO_to_best_n_term}]
We first assume that $(\mathbf{A},\Delta)\in\operatorname{IO}(q,p,n,C_N)$. Let $\mathbf{z}\in\ker(\mathbf{A})$ and $S$ be the index set of the $n$ largest absolute entries of $\mathbf{z}$. Then the instance optimality implies that $-\mathbf{z}_S=\Delta(-\mathbf{A}\mathbf{z}_S)$, and, since $\mathbf{z}\in\ker(\mathbf{A})$, we further get $-\mathbf{z}_S=\Delta(\mathbf{A}\mathbf{z}_{S^\mathrm{C}})$. Using again $(\mathbf{A},\Delta)\in\operatorname{IO}(q,p,n,C_N)$, we obtain the claim \eqref{eqn:IOqNorm} from
\begin{align}
\lVert\mathbf{z}_{S^\mathrm{C}}+\mathbf{z}_S\rVert_{\ell_q}=\lVert\mathbf{z}_{S^\mathrm{C}}-\Delta(\mathbf{A}\mathbf{z}_{S^\mathrm{C}})\rVert_{\ell_q} \leq C_Nn^{\sfrac{1}{q}-\sfrac{1}{p}}\sigma_n(\mathbf{z}_{S^\mathrm{C}})_{\ell_p}= C_Nn^{\sfrac{1}{q}-\sfrac{1}{p}}\sigma_{2n}(\mathbf{z})_{\ell_p}.
\end{align}
We now assume that \eqref{eqn:IOqNorm} holds and define a reconstruction map $\Delta\colon \C^m \to\C^N$ via
\begin{align}\label{eq:DefDelta}
\Delta(\mathbf{y}) \in \argmin \bigl\{\sigma_n(\mathbf{v})_{\ell_p}\colon \; \mathbf{v}\in\mathbb{C}^N \text{ with } \mathbf{A}\mathbf{v}=\mathbf{y}\bigr\}.
\end{align}
For any $\mathbf{z}\in\mathbb{C}^N$, applying \eqref{eqn:IOqNorm} to $\mathbf{z}-\Delta(\mathbf{A}\mathbf{z})\in\ker(\mathbf{A})$ gives
\begin{align}
\lVert\mathbf{z}-\Delta(\mathbf{A}\mathbf{z})\rVert_{\ell_q} \leq C_Nn^{\sfrac{1}{q}-\sfrac{1}{p}}\sigma_{2n}(\mathbf{z}-\Delta(\mathbf{A}\mathbf{z}))_{\ell_p}.
\end{align}
Now, $\sigma_{2n}(\mathbf{u}+\mathbf{v})_{\ell_p}\leq\sigma_n(\mathbf{u})_{\ell_p}+\sigma_n(\mathbf{v})_{\ell_p}$ and the definition of $\Delta$ in \eqref{eq:DefDelta} yield
\begin{align}
\lVert\mathbf{z}-\Delta(\mathbf{A}\mathbf{z})\rVert_{\ell_q}\leq C_Nn^{\sfrac{1}{q}-\sfrac{1}{p}}\bigl(\sigma_{n}(\mathbf{z})_{\ell_p}+\sigma_n(\Delta(\mathbf{A}\mathbf{z}))_{\ell_p}\bigr)\leq 2C_Nn^{\sfrac{1}{q}-\sfrac{1}{p}}\sigma_{n}(\mathbf{z})_{\ell_p},
\end{align}
namely that $(\mathbf{A},\Delta)\in\operatorname{IO}(q,p,n,2C_N)$.
\end{proof}

\begin{proof}[Proof of \Cref{thm:IOPreservation}]
Let $\mathbf{z}\in\ker(\mathbf{A})$ be arbitrary.
We denote the set of the $3n$ largest absolute entries of $\mathbf{z}$ by $S$.
From \Cref{prop:IO_to_best_n_term} we get
\begin{align}\label{eq:Estlq1}
\begin{split}
\lVert\mathbf{z}_S\rVert_{\ell_{q'}}&\leq(3n)^{\sfrac{1}{q'}-\sfrac{1}{q}}\lVert\mathbf{z}_S\rVert_{\ell_q}\leq(3n)^{\sfrac{1}{q'}-\sfrac{1}{q}}\lVert\mathbf{z}\rVert_{\ell_q}\leq(3n)^{\sfrac{1}{q'}-\sfrac{1}{q}} C_Nn^{\sfrac{1}{q}-\sfrac{1}{p}}\sigma_{2n}(\mathbf{z})_{\ell_p}\\
&=3^{\sfrac{1}{q'}-\sfrac{1}{q}}C_N n^{\sfrac{1}{q'}-\sfrac{1}{p}}\sigma_{2n}(\mathbf{z})_{\ell_p}\leq 3C_N n^{\sfrac{1}{q'}-\sfrac{1}{p}}\sigma_{2n}(\mathbf{z})_{\ell_p}.
\end{split}
\end{align}
Moreover, \Cref{thm:Sigmas} yields
\begin{gather}\label{eq:Estlq2}
    \lVert\mathbf{z}_{S^\mathrm{C}}\rVert_{\ell_{q'}}=\sigma_{3n}(\mathbf{z})_{\ell_{q'}}\leq n^{\sfrac{1}{q'}-\sfrac{1}{p}}\sigma_{2n}(\mathbf{z})_{\ell_p}.
\end{gather}
Combining \eqref{eq:Estlq1} and \eqref{eq:Estlq2}, we infer that
\begin{gather}
   \lVert\mathbf{z}\rVert_{\ell_{q'}}\leq\lVert\mathbf{z}_S\rVert_{\ell_{q'}}+\lVert\mathbf{z}_{S^\mathrm{C}}\rVert_{\ell_{q'}}\leq(3C_N+1)n^{\sfrac{1}{q'}-\sfrac{1}{p}}\sigma_{2n}(\mathbf{z})_{\ell_p}. 
\end{gather}
Using again \Cref{prop:IO_to_best_n_term}, we thus get $(\mathbf{A},\Delta)\in\operatorname{IO}(q^\prime,p,n,6C_N + 2)$.
\end{proof}

\begin{proof}[Proof of \Cref{thm:LowerBoundSingleNorm}]
We consider the family $\mathcal{U}$ from \Cref{thm:SubsetCombinations}.
For all $S\in\mathcal{U}$, we define $\mathbf{z}^S$ via $z^S_k\coloneqq n^{-\sfrac{1}{p}}\boldsymbol{1}_S(k)$, for which we have \smash{$\lVert\mathbf{z}^S\rVert_{\ell_p}=1$}.
Next, we show that the sets \smash{$\mathbf{A}(\mathbf{z}^S+\varphi\mathbb{B}_p)$} with $\varphi=(2C_N+2)^{-1}$ are disjoint, where $\mathbb{B}_p$ denotes the $\ell_p$-unit ball in $\mathbb{C}^N$. To this end, let $S,T\in\mathcal{U}$ with $S\neq T$, and assume that there exist $\mathbf{z},\mathbf{z}'\in\varphi\mathbb{B}_{p}$ with $\mathbf{A}(\mathbf{z}^S+\mathbf{z})=\mathbf{A}(\mathbf{z}^T+\mathbf{z}')$.
From $\lvert S\cap T\rvert<\sfrac{n}{2}$ we infer that
\begin{align}
\begin{split}
1& < \bigl\lVert \mathbf{z}^S -\mathbf{z}^T \bigr\rVert_{\ell_p}=\bigl\lVert\mathbf{z}^S + \mathbf{z} - \Delta\bigl(\mathbf{A}(\mathbf{z}^S+\mathbf{z})\bigr)+\Delta\bigl(\mathbf{A}(\mathbf{z}^T+\mathbf{z}')\bigr) -\mathbf{z}'-\mathbf{z}^T-\mathbf{z}+\mathbf{z}'\bigr\rVert_{\ell_p}\\
&\leq\big\lVert\mathbf{z}^S+\mathbf{z}-\Delta\bigl(\mathbf{A}(\mathbf{z}^S+\mathbf{z})\bigr)\big\rVert_{\ell_p} + \big\lVert\mathbf{z}^T+\mathbf{z}'-\Delta\bigl(A(\mathbf{z}^T+\mathbf{z}')\bigr)\big\rVert_{\ell_p}+ \lVert\mathbf{z}\rVert_{\ell_p} + \lVert\mathbf{z}'\rVert_{\ell_p}.
\end{split}
\end{align}
Since $(\mathbf{A},\Delta) \in\operatorname{IO}(p,n,C_N)$ and $\vert S \vert = \vert T \vert = n$, we further have
\begin{align}
\begin{split}
1 & < C_N\sigma_n\big(\mathbf{z}^S + \mathbf{z}\big)_{\ell_p} + C_N\sigma_n \big(\mathbf{z}^T + \mathbf{z}'\big)_{\ell_p} + \lVert \mathbf{z} \rVert_{\ell_p} + \lVert \mathbf{z}' \rVert_{\ell_p}\\
&\leq C_N \lVert\mathbf{z} \rVert_{\ell_p}+C_N \lVert \mathbf{z}' \rVert_{\ell_p} + \lVert \mathbf{z} \rVert_{\ell_p}+ \lVert \mathbf{z}' \rVert_{\ell_p} \leq (2C_N+2)\varphi=1.
\end{split}
\end{align}
This is a contradiction, and thus the sets \smash{$\mathbf{A}(\mathbf{z}^S+\varphi\mathbb{B}_{p})$} are disjoint.
Since \smash{$\lVert\mathbf{z}^S+\mathbf{z}\rVert_{\ell_p} \leq 1 + \lVert\mathbf{z}\rVert_{\ell_p}$}, we know that $\mathbf{z}^S+\varphi\mathbb{B}_{p}\subset\left(1+\varphi\right)\mathbb{B}_{p}$ for all $S\in\mathcal{U}$, and consequently also
\begin{equation}\label{eq:EstimateSets}
    \mathbf{A}(\mathbf{z}^S+\varphi\mathbb{B}_{p})\subset\left(1+\varphi\right)\mathbf{A}(\mathbb{B}_{p}).
\end{equation}
We continue with a volumetric argument. First, we define $r\coloneqq\operatorname*{rank}(\mathbf{A})$ and recall that the complex $r$-dimensional balls have the same volume as the $2r$-dimensional real balls. Since the sets $\mathbf{A}(\mathbf{z}^S+\varphi\mathbb{B}_{p})$ are disjoint, we have by \eqref{eq:EstimateSets}
\begin{align}
\begin{split}
\lvert\mathcal{U}\rvert\varphi^{2r}\operatorname*{vol}(\mathbf{A}(\mathbb{B}_{p}))&=\lvert\mathcal{U}\rvert \operatorname*{vol}(\mathbf{A}(\varphi \mathbb{B}_{p}))=\sum\limits_{S\in\mathcal{U}}\operatorname*{vol}\bigl(\mathbf{A}(\mathbf{z}^S+\varphi\mathbb{B}_{p})\bigr)\\
&\leq\operatorname*{vol}\left(\left(1+\varphi\right)\mathbf{A}(\mathbb{B}_{p})\right)=\left(1+\varphi\right)^{2r}\operatorname*{vol}(\mathbf{A}(\mathbb{B}_{p})).
\end{split}
\end{align}
Thus, we see that $\lvert\mathcal{U}\rvert\varphi^{2r}\leq\left(\varphi+1\right)^{2r}$.
This further implies that $\lvert\mathcal{U}\rvert \leq (2C_N+3)^{2r}$.
Since $r\leq m$ and $\lvert\mathcal{U}\rvert\geq(\sfrac{N}{4n})^{\sfrac{n}{2}}$, we then obtain $(\sfrac{N}{4n})^{\sfrac{n}{2}} \leq(2C_N+3)^{2m}$.
Finally, applying the logarithm on both sides leads to
\begin{gather}
    m\geq\frac{n\log(\sfrac{N}{4n})}{4\log(2C_N+3)}.
\end{gather}
\end{proof}